\documentclass[12pt,DIV=18]{scrartcl}

\usepackage{natbib}
\usepackage{amsmath,amssymb}
\usepackage{amsthm,MnSymbol,xcolor}

\usepackage{setspace}

  \usepackage{graphicx,hyperref}
\usepackage{setspace}

\usepackage{color}
\usepackage{array}
\usepackage{supertabular}
\usepackage{hhline}
\usepackage{multirow}
\usepackage{multicol}
\usepackage{textgreek}
\usepackage{upgreek}

\usepackage{wrapfig}

\usepackage{mathtools}

\usepackage{stmaryrd}

\usepackage[listings,theorems]{tcolorbox}

\usepackage{pifont}

\usepackage[thicklines]{cancel}

\usepackage[LGR,T1]{fontenc}

\usepackage{tikz}

\usetikzlibrary{hobby}

\usepackage[customcolors]{hf-tikz}

\usepackage{circuitikz}

\usetikzlibrary{decorations.pathreplacing,arrows.meta}

\usetikzlibrary{calc,patterns,decorations.markings}
\usetikzlibrary{fadings}
\usetikzlibrary{angles,quotes}
\usepackage{tikz-3dplot}

\newtheorem{theorem}{Theorem}[section]

\newtheorem{remark}[theorem]{Remark}
\newtheorem{definition}{Definition}

\begin{document}

\begin{center}
{\Large \sffamily \bfseries Stretched non-local Pearson diffusions}
\end{center}

\begin{center}
{Luisa Beghin}$^{\textrm{a},}$\footnote{Corresponding author.}, {Nikolai Leonenko}$^{\textrm{b}}$, {Ivan Papić}$^{\textrm{c}}$ and
{Jayme Vaz}$^{\textrm{d}}$ 

\footnotesize{
		$$\begin{tabular}{llll}
$^{a}$Department of Statistical Sciences, Sapienza University, p.le Aldo Moro 5, Rome, 00185, Italy. \\
$^{b}$Cardiff School of Mathematics, Cardiff University, Senghennydd Road,
Cardiff, CF24 4AG, UK.\\
$^{c}$School of Applied Mathematics and Informatics, University of Osijek, Trg Ljudevita Gaja 6,
Osijek, Croatia.\\
$^{d}$Departamento de Matem\'atica Aplicada, 
Universidade Estadual de Campinas, 13087-859, Campinas, SP, Brazil. 
\end{tabular}$$}
\end{center}
\let\thefootnote\relax\footnotetext{E-mail adresses: luisa.beghin@uniroma1.it (Luisa Beghin), LeonenkoN@cardiff.ac.uk (Nikolai Leonenko), ipapic@mathos.hr (Ivan Papić), vaz@unicamp.br (Jayme Vaz).}
\noindent\textbf{Abstract.}
We define a novel class of time-changed Pearson diffusions, termed stretched non-local Pearson diffusions, where the stochastic time-change model has the Kilbas–Saigo function as its Laplace transform. Moreover, we introduce a stretched variant of the Caputo fractional derivative and prove that its eigenfunction is, in fact, the Kilbas–Saigo function. Furthermore, we solve fractional Cauchy problems involving the generator of the Pearson diffusion and the Fokker–Planck operator, providing both analytic and stochastic solutions, which connect the newly defined process and fractional operator with the Kilbas–Saigo function. We also prove that stretched non-local Pearson diffusions share the same limiting distributions as their standard counterparts. Finally, we investigate fractional hyperbolic Cauchy problems for Pearson diffusions, which resemble time-fractional telegraph equations, and provide both analytical and stochastic solutions. As a byproduct of our analysis, we derive a novel representation and an asymptotic formula for the Kilbas–Saigo function $\operatorname{E}_{a,m,l}(z)$ with complex $z$, which, to the best of our knowledge, are not currently available in the existing literature.
\bigskip


\bigskip


\bigskip
 
\section{Introduction}
Pearson diffusions  \cite{Pearson,Forman,Meerschaert} are time-homogeneous Markov processes $(X_t, \, t \geq 0)$ with transition density $\boldsymbol{p}(x,t;y) = 
\frac{\partial}{\partial x}P(X_t\leq x|X_0 = y)$ characterized by diffusion equations with polynomial coefficients. The function 
$p(x,t) = \int \boldsymbol{p}(x,t;y) f(y)\, dy $ satisfies the 
Fokker-Planck equation (also known as the Kolmogorov forward equation)  
\begin{equation*}
\label{FP.eq}
\frac{\partial p(x,t)}{\partial t} = \mathcal{L}[p] = - \frac{\partial\;}{\partial x} 
[\mu(x)p(x,t)] + \frac{\partial^2\;}{\partial x^2}[D(x) p(x,t)] , 
\end{equation*}
with the diffusion coefficient    
\begin{equation*}
\label{dif.coeff}
D(x) = d_0 + d_1 x + d_2 x^2
\end{equation*}
and the drift 
\begin{equation*}
\mu(x) = a_0 + a_1 x ,
\end{equation*}
and initial condition $p(x,0) = \delta(x-y)$. Here, $\mathcal{L}$ denotes Fokker-Planck operator. On the other 
hand, the function 
\begin{equation}
\label{eq.g.def}
g(y,t) = \int \boldsymbol{p}(x,t;y) g(x)\, dx
\end{equation}
satisfies 
the Kolmogorov backward equation 
\begin{equation}
\label{kolmo}
\frac{\partial g(y,t)}{\partial t} = \mathcal{G}[g] = \mu(y)\frac{\partial g(y,t)}{\partial y} + 
D(y) \frac{\partial^2 g(y,t)}{\partial y^2}   
\end{equation}
with the condition $g(y,0) = g(y)$, where $\mathcal{G}$ is infinitesimal generator of the corresponding diffusion (see e.g. \cite{Karlin}). 

In general, classical diffusion models rely on assumptions of Markovianity that are often violated in complex systems exhibiting long-memory effects and non-local interactions. To address these limitations, fractional calculus has become an increasingly important tool, introducing non-local temporal dynamics through derivatives of non-integer order (for recent developments on this topic see e.g. \cite{abdelhamid2023existence}, \cite{ayad2024nonlinear}, \cite{dhanalakshmi2023exponential}, \cite{gorska2023subordination}, \cite{tuan2021time}).

In \cite{Pearson}, the authors explore specific fractional diffusion model, so called Fractional Pearson diffusions. The latter involve diffusion equations governed by a time-fractional diffusion equation. Specifically, in place of \eqref{kolmo}, 
they investigated the equation 
\begin{equation}
\label{frac.kolmo}
{\sideset{_{\scriptscriptstyle C}^{}}{_t^{(\alpha)}}{\operatorname{\mathcal{D}}}} g = 
\mathcal{G}[g] ,
\end{equation}
where ${\sideset{_{\scriptscriptstyle C}^{}}{_t^{(\alpha)}}{\operatorname{\mathcal{D}}}}$ 
is the Caputo fractional derivative of order $\alpha$ (see e.g. \cite{podlubny1998fractional}) and 
$\mathcal{G}[g]$ is as in \eqref{kolmo}. 
This leads to solutions expressible in terms of Mittag-Leffler functions (\cite{MittagLeffler}), and corresponds to time-changing the original diffusion by the inverse of a stable subordinator.

In this work, we extend the fractional Pearson diffusion framework by introducing a new class of stochastic processes, termed stretched non-local Pearson diffusions, governed by a fractional differential operator of the form
$$\frac{1}{t^{\gamma}}{\sideset{_{\scriptscriptstyle C}^{}}{_t^{(\alpha)}}{\operatorname{\mathcal{D}}}},$$
where $\gamma$ is a stretching parameter. This operator, a time-stretched variant of the Caputo derivative, leads to time evolutions involving Kilbas–Saigo function (\cite{Simon}), which generalize the Mittag-Leffler function and exhibit a broader range of decay behaviors and asymptotics. Moreover, the stochastic solution is time-changed Pearson diffusion where time-change model is more general than the inverse subordinator.

Conversely, hyperbolic diffusion equations play a crucial role when there is a necessity to establish a bounded propagation speed for diffusion phenomena. An example of such an equation is the Cattaneo equation \cite{Cattaneo}, which exhibits formal resemblance to the linear telegraph equation (\cite{Telegraph_eq}). Given this, it's natural to extend the results of \cite{Pearson} by incorporating the realm of hyperbolic diffusions.
This broader perspective allows us to explore a wider range of possibilities and mathematical techniques. Incorporating fractional time derivatives (or, more generally, non-local temporal operators) into telegraph equations often results in significantly improved model fitting to empirical data compared to classical telegraph-type time operators (see \cite{madhukar2019heat}, \cite{ran2024heat} and \cite{ran2022shock}).

Structure of the paper is as follow. In Section \ref{KS.section} we define new fractional operator with stretching parameter $\gamma$ which generalizes Caputo fractional derivative. We show that Kilbas-Saigo (KS) function is eigenfunction for this newly defined operator. We explore bounds and asymptotic behaviour for KS function for complex arguments. Moreover, we consider a second-order equation involving this fractional operator which resembles time-part of telegraph equation and calculate its solution in terms of Kilbas-Saigo function.  
Section \ref{PearsonDiffusions} contains overview of Pearson diffusions classification together with their spectral properties which we use in the follow up sections.
In Section \ref{sec_fpd} we provide analytical solutions for fractional Cauchy problems with stretching involving Pearson diffusions in terms of Kilbas-Saigo function.
In Section \ref{sec_fhpd} we consider general fractional Cauchy problems with stretching, where time-part resembles fractional telegraph equation, while space part is driven by Pearson diffusion infinitesimal parameters. Analytical solutions are written in terms of Kilbas-Saigo function with complex arguments.
Finally, Section \ref{Stoch_rep_sec} explores stochastic representation results of analytical solutions of previous two sections.
Appendix A contains overview of properties of double gamma function which we extensively use in Section \ref{KS.section}.

\section{The Kilbas-Saigo function and the differential operator $\boldsymbol{\mathcal{D}^{(\alpha,\gamma)}_t}$ }  \label{KS.section}

\begin{definition}
	We define the non-local differential operator $\mathcal{D}^{(\alpha,\gamma)}_t $ as
	\begin{equation}
		\label{Caputo_variation}
		\mathcal{D}^{(\alpha,\gamma)}_t f(t) := 
		\frac{t^{-\gamma}}{\Gamma(1-\alpha)}\int_0^t \frac{f^\prime(\tau)}{\left(t-\tau\right)^{\alpha}}\, d\tau , =	t^{-\gamma}
		\sideset{_{\scriptscriptstyle C}^{}}{_t^{(\alpha)}}{\operatorname{\mathcal{D}}},
	\end{equation}
	where $0 < \alpha <1$ is order of the Caputo fractional derivative $\sideset{_{\scriptscriptstyle C}^{}}{_t^{(\alpha)}}{\operatorname{\mathcal{D}}}$ and $\gamma$ is an arbitrary nonnegative real number, representing stretching parameter of time argument.
\end{definition}

\begin{remark}
	In what follows we will extensively use these Caputo-type derivatives of power law functions, so let us recall that 
	\begin{equation*}
		{\sideset{_{\scriptscriptstyle C}^{}}{_t^{(\alpha)}}{\operatorname{\mathcal{D}}}} t^\beta  = 
		\begin{cases}
			{\displaystyle \frac{\Gamma(\beta+1)}{\Gamma(\beta-\alpha+1)} t^{\beta-\alpha}} , \; & \beta \neq 0 , \\[1ex]
			0 , & \beta = 0 ,
		\end{cases}
	\end{equation*}
	and then 
	\begin{equation*}
		\operatorname{\mathcal{D}}^{(\alpha,\gamma)}_t  \, t^\beta 
		= 
		\begin{cases}
			{\displaystyle \frac{\Gamma(\beta+1)}{\Gamma(\beta-\alpha+1)} t^{\beta-(\alpha+\gamma)}} , \; & \beta \neq 0 , \\[1ex]
			0 , & \beta = 0 .
		\end{cases}
	\end{equation*}
\end{remark}

The Kilbas-Saigo function (see e.g. \cite{Simon}), denoted by $\operatorname{E}_{a,m,l}(z)$, is defined as 
\begin{equation*}
	\label{def.KS.function}
	\operatorname{E}_{a,m,l}(z) = \sum_{n=0}^\infty c_n z^n ,  
\end{equation*}
with 
\begin{equation}
	\label{def.coeff.KS.function}
	c_0 = 1 , \qquad c_n = \prod_{k=0}^{n-1} \frac{\Gamma[1+a(km+l)]}{\Gamma[1+a(km+l+1)]} , 
\end{equation}
and where the parameters are such that $a, m > 0$ and $l > -1/a$.
The KS function $\operatorname{E}_{a,m,m-1}(-x)$ is monotonically decreasing for $x \geq 0$ and satisfies \cite{Simon} 
\begin{equation} \label{KS_function_bounds}
\frac{1}{1+\Gamma(1-a)x} \leq \operatorname{E}_{a,m,m-1}(-x)\leq \frac{1}{1+  \displaystyle {\frac{\Gamma(1+a(m-1))}{\Gamma(1+am)}} x} ,
\end{equation}
where $x \geq 0$ and $a \in [0,1]$. 
Moreover, the following asymptotic representation is valid:
\begin{equation} \label{KS_realz_asympt}
	\operatorname{E}_{a,m,l}(-x) \sim  \frac{\Gamma(1+a(l+1-m))}{\Gamma(1+a(l-m))x}, \quad \text{ as } x \to \infty,
\end{equation}
where $a \in \langle 0, 1\rangle, m>0$, $l > m-1/a$. \\

In particular, for $a \in \langle 0, 1 \rangle$ and $m>0$:
\begin{align}
1-\operatorname{E}_{a,m,m-1}(-x)& \sim  \frac{\Gamma(1+a(m-1))}{\Gamma(1+am)x}, \quad \text{ as } x \to 0, \\  \label{KS_realz_asympt_spec}\operatorname{E}_{a,m,m-1}(-x) &\sim  \frac{1}{\Gamma(1-a)x}, \quad \text{ as } x \to \infty.
\end{align}

However, it turns out that our analysis will require similar asymptotics but for complex argument $x$ with positive real part. To the best of our knowledge such results are not available in literature, and hopefully our results covers this gap.

In what follows we will make an extensive use of the double gamma function, since we will rewrite the Kilbas-Saigo function using it and then use results from the double gamma function to obtain asymptotic results for the Kilbas-Saigo function. The double gamma function is generalization of the gamma function \cite{Genesis}, but since it is not as well known as other functions, it is beneficial to recall its definition and main properties (see \ref{Appen}).

In fact, KS function can be rewritten in terms of double gamma function. We can rewrite the $c_k$ coefficients in \eqref{def.coeff.KS.function} using \eqref{ap.B.general.G.tau.n}
in order to obtain: 
\begin{equation*}
	\label{coeff.KS.G.pochhammer}
	c_k = \frac{G(\varphi + a\tau;\tau)}{G(\varphi ;\tau)} \, \frac{G(\varphi   + k;\tau)}{G(\varphi + a\tau+ k;\tau)} , 
\end{equation*}
where 
\begin{equation}
	\label{def.varphi}
	\varphi = (1+al)\tau 
\end{equation}
and 
\begin{equation*}
	\label{exp.tau}
	\tau = \frac{1}{am}  
\end{equation*} 
for $k=1,2,\ldots$. 
Note that this expression also works for $k=0$. 
Thus we can write 
\begin{equation}
	\label{def.KS.G.pochhammer}
	\operatorname{E}_{a,m,l}(z) = \frac{G(\varphi + a\tau;\tau)}{G(\varphi ;\tau)} \sum_{n=0}^\infty  \frac{G(\varphi   + n;\tau)}{G(\varphi + a\tau+ n;\tau)} z^n . 
\end{equation}
Next theorem provides Mellin-Barnes type representation for KS function which is needed for the following asymptotic result.
\begin{theorem}
Let $z \in \mathbb{C}$ be such that $\operatorname{Re}(z)>0$. Then the following formula is valid and integral on the right-hand side is convergent for any such $z$:
\begin{equation}
	\label{M-B.KS.1}
\operatorname{E}_{a,m,l}(-z)=	\frac{1}{2\pi i}\frac{G(\varphi + a\tau;\tau)}{G(\varphi ;\tau)}   \int_{C} \frac{\Gamma(s)\Gamma(1-s)G(\varphi -s ;\tau)}{G(\varphi+a\tau -s;\tau)} z^{-s} \, ds ,
\end{equation}
with $\tau = 1/(am) > 0$. 
\end{theorem}

\begin{proof}
Due to the gamma functions in the numerator, the integrand has poles at $s=-n$ ($n=0,1,2,\ldots$) from $\Gamma(s)$ and at $s = n+1$ ($n=0,1,2,\ldots$) from $\Gamma(1-s)$. Since the double gamma function is entire, the other poles of the integrand come from the zeros of $G(\varphi+a\tau -s;\tau)$ 
in the denominator.  The zeros from $G(\varphi+a\tau -s;\tau)$ are located at $s = \varphi + a\tau + \mu\tau + \lambda$ ($\lambda,\mu=0,1,2,\ldots$). We will choose the contour of integration $C$ such that it separates the poles of $\Gamma(s)$  from the poles of $\Gamma(1-s)$ and $1/G(\varphi+a\tau -s;\tau)$. Thus $C = (c-i\infty,c+i\infty)$ where $0 < c < \varepsilon$ where $\varepsilon =\operatorname{min}[1,(1+a(l+1))\tau]$.

In order to analyze the convergence of this integral, we will consider 
the limit $N \to \infty$ for the integral along $(c-i(N+1/2),c+i(N+1/2))$ 
(the dashed line segment in Figure~\ref{fig.2}). 
However, since the integrand is a holomorphic function in the 
entire plane except at the poles described above, as usual we deform the line segment  
$(c-i(N~+~1/2), c+i(N+1/2))$  into the contour $C_N^- \, \cup \,  C_\epsilon  \, \cup \,  C_N^+ $, where 
$C_N^- = (-i(N~+~1/2),-\epsilon)$, $C_N^+ = (\epsilon,i(N+1/2))$ and 
$C_\epsilon$ is the arc $\{s \in \mathbb{C}\,|\,
|s| = \epsilon, -\pi/2 \leq \operatorname{arg}(s) \leq \pi/2\}$ 
encircling the pole $s=0$ from the right (see Figure~\ref{fig.2}). 

\begin{figure}[h]
	\begin{center}
		\begin{tikzpicture}[scale=0.8]
			\draw[->] (-6,0) -- (2,0) node[below] {${\scriptscriptstyle \operatorname{Re}s}$};
			\draw[->] (0,-4.7) -- (0,4.8) node[right] {${\scriptscriptstyle \operatorname{Im}s}$};
			\draw[orange,fill=orange] (0,0) circle (1.2pt);
			\draw[orange,fill=orange] (-1,0) circle (1.2pt);
			\draw[orange,fill=orange] (-2,0) circle (1.2pt);
			\draw[orange,fill=orange] (-4,0) circle (1.2pt);
			\draw[orange,fill=orange] (-5,0) circle (1.2pt);
			\node[below] at (1,0) {${\scriptscriptstyle \varepsilon}$};
			\node[below] at (-1,0) {${\scriptscriptstyle -1}$};
			\node[below left] at (0,0) {${\scriptscriptstyle 0}$};
			\node[below] at (-2,0) {${\scriptscriptstyle -2}$};
			\node[below] at (-3,0) {$\cdots$};
			\node[below] at (-4,0) {${\scriptscriptstyle -N}$};
			\node[below] at (-5.2,0) {${\scriptscriptstyle -(N+1)}$};
			\draw[cyan,very thick] (0,0.5) -- (0,4.5);
			\draw[very thick,dashed] (0.8,-4.5) -- (0.8,4.5);
			\draw[cyan,very thick] (0,-0.5) -- (0,-4.5);
			\draw[->,>=triangle 45,cyan,very thick] (0,1) -- (0,1.9);
			\draw[cyan, very thick] (0,4.5) arc (90:270:4.5);
			\draw[cyan, very thick,->,>=triangle 45] (0,4.5) arc (90:160:4.5);
			\draw[cyan, very thick] (0,-0.5) arc (-90:90:0.5);
			\node[below] at (0.5,-0.4) {\footnotesize ${\textcolor{cyan}{C_\epsilon }}$};
			\node[left] at (0,3) {\footnotesize ${\textcolor{cyan}{C_N^+ }}$};
			\node[left] at (0,-3) {\footnotesize ${\textcolor{cyan}{C_N^-}}$};
			\node[right] at (0.8,-4.5) {$c-i(N+1/2)$};
			\node[right] at (0.8,4.5) {$c+i(N+1/2)$};
			\node[below, cyan] at (-5,2.6) {$C_N$};
			\draw[->] (0,0) -- (-3.15,3.15);
			\node[right] at (-2,2) {${\scriptscriptstyle N+1/2}$};
		\end{tikzpicture}
	\end{center}
	\caption{Integration contour for \eqref{M-B.KS.1}.\label{fig.2}}
\end{figure}
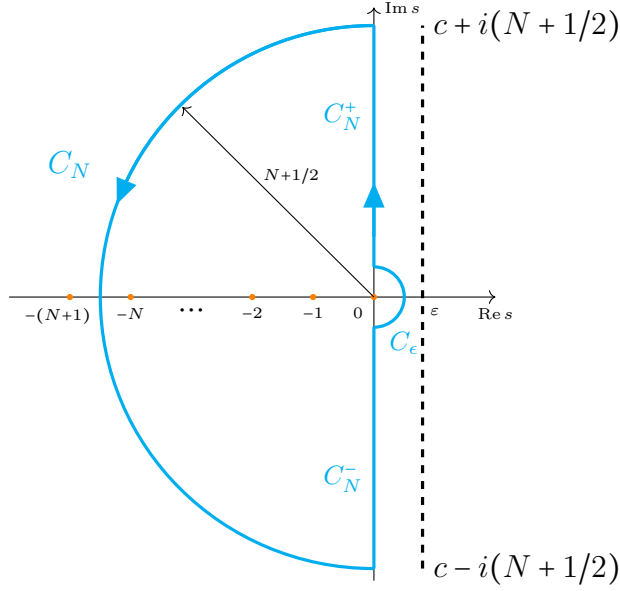

Since the integral along the arc $C_\epsilon$ vanishes for $\epsilon \to 0$, 
we only need to analyse the integrals along $C_N^-$ and $C_N^+$. We can 
write $s \in C_N^\pm$ as $s = R\operatorname{e}^{i\theta}$ with 
$\epsilon \leq R \leq (N+1/2)$ and $\theta = \pi/2$ for $s \in C_N^+$ and
$\theta = -\pi/2$ for $s \in C_N^-$. Let us analyze 
\begin{equation*}
	\label{log.integrand}
	\log |\mathcal{I}(s)| = 
	\log\left|
	\frac{\Gamma(s)\Gamma(1-s)G(\varphi -s;\tau)}{G(\varphi+a\tau -s;\tau)} z^{-s} \right| . 
\end{equation*}
We have  
\begin{equation*}
	\log|z^{-s}| = -R\cos\theta \log|z| + 
	R\sin\theta \operatorname{arg}(z) , 
\end{equation*}
from Stirling formula, 
\begin{gather*}
	\log|\Gamma(s)| = \left(R\cos\theta -\frac{1}{2}\right) \log{R} - 
	R(\theta\sin\theta + \cos\theta) + \mathcal{O}(1)  , \\
	\log|\Gamma(1-s)| = -\left(R\cos\theta -\frac{1}{2}\right) \log{R} + 
	R(\theta\sin\theta - \pi \sin|\theta| + \cos\theta) + \mathcal{O}(1) , 
\end{gather*}
and from \eqref{ap.B.Stirling.G.tau}, 
\begin{equation*}
	\begin{split}
		& \log|G(1+c-s;\tau)| =   \left(A_2(\tau,c)\cos{2\theta}\right) R^2\log{R} - \left(A_1(\tau,c)\cos\theta\right) R\log{R} \\ 
		& \qquad  +A_0(\tau,c) \log{R} + \left[B_2(\tau,c)\cos{2\theta} - 
		A_2(\tau,c)\left( \theta\sin{2\theta} -\pi \sin{2|\theta|}\right)\right] R^2 \\
		& \qquad  -\left[B_1(\tau,c)\cos\theta-A_1(\tau,c)\left( \theta\sin\theta -\pi \sin|\theta|\right)\right] R 
		+ \mathcal{O}(1) , 
	\end{split}
\end{equation*}
where 
\begin{equation*}
	\begin{split}
		& A_2(\tau,c) = a_2(\tau) , \\
		& A_1(\tau,c) = 2(1+c)a_2(\tau) + a_1(\tau), \\
		& A_0(\tau,c) = (1+c)^2 a_2(\tau) + (1+c) a_1(\tau) + a_0(\tau) , \\
		& B_2(\tau,c) = b_2(\tau) , \\
		& B_1(\tau,c) = (1+c)a_2(\tau) + 2(1+c) b_2(\tau) + b_1(\tau) .
	\end{split}
\end{equation*}
Using these expressions in \eqref{log.integrand} we obtain 
\begin{equation}
	\label{exp.log.integrand}
	\log\left|\mathcal{I}(s) \right|  
	= K_2(\tau,a) \, R\log{R} + K_1(\tau,a) \, R + K_0(\tau,a) \, \log{R} + \mathcal{O}(1) ,
\end{equation}
where
\begin{equation*}
	\begin{aligned}
		& K_2(\tau,a) = a \cos\theta, \\
		& \begin{aligned} K_1(\tau,a) = & -\cos\theta\, \log|z| + \sin\theta\, \operatorname{arg}z - \pi \sin|\theta| \\
			& -a(1+\log\tau) \cos\theta - a \left(\theta\sin\theta - \pi \sin|\theta|\right) ,
		\end{aligned} \\
		& K_0(\tau,a) = \frac{a}{2}[1-\tau\left( 1+a(2l+1)\right)] .
	\end{aligned}
\end{equation*}

Let us see what happens with the integrals along $C_N^+$ and $C_N^-$. 
Along $C_N^+$ we have $\theta = \pi/2$, and so  
\begin{equation*}
	\log\left|\mathcal{I}(s) \right| \big|_{C_N^+} = \left[\operatorname{arg}z -\pi - a\left(\frac{\pi}{2} - \pi\right)\right] R + \mathcal{O}(\log{R}) .
\end{equation*}
Thus if 
\begin{equation*}
	\operatorname{arg}z < \left(1 - \frac{a}{2}\right)\pi
\end{equation*}
the integral along $C_N^+$ decays exponentially as $N \to \infty$.  
On the other hand, along $C_N^-$ we have $\theta = -\pi/2$, and so 
\begin{equation*}
	\log\left|\mathcal{I}(s) \right| \big|_{C_N^-} = \left[-\operatorname{arg}z -\pi - a\left(\frac{\pi}{2} - \pi\right)\right] R + \mathcal{O}(\log{R}) .
\end{equation*}
Thus if 
\begin{equation*}
	\operatorname{arg}z > -\left(1 - \frac{a}{2}\right)\pi
\end{equation*}
the integral along $C_N^-$ decays exponentially as $N \to \infty$.
Consequently, the integral along $C$ converges in the sector 
\begin{equation*}
	|\operatorname{arg}z| < \left(1 - \frac{a}{2}\right)\pi .
\end{equation*}
Since $0 < a < 1$, we conclude that integral in \eqref{M-B.KS.1} converges for $	\operatorname{Re} z > 0.$ 

Next, we prove the equality in \eqref{M-B.KS.1}.
We will consider the closed contour $C_N^- \, \cup \,  C_\epsilon  \, \cup \,  C_N^+ \, \cup \,  C_N$, where $C_N$ is the arc of radius $N + 1/2$ on the left half-plane illustrated in 
Figure~\ref{fig.2}. For $s \in C_N$ we have $\cos\theta < 0$ and for 
the leading term in \eqref{exp.log.integrand} we have  
$$
a \cos\theta R \log{R} < 0 , 
$$
which implies that  
$$
\lim_{N\to \infty} \int_{C_N} \frac{\Gamma(s)\Gamma(1-s)G(\varphi  -s;\tau)}{G(\varphi+a\tau -s;\tau)} z^{-s} ds = 0 . 
$$
Thus, from the residue theorem, 
\begin{equation*}
	\begin{aligned}
		f(z) & =  \frac{G(\varphi+a\tau;\tau)}{G(\varphi;\tau)} \sum_{n=0}^\infty \underset{s=-n}{\operatorname{Res}} \left[\frac{\Gamma(s)\Gamma(1-s)G(\varphi  -s;\tau)}{G(\varphi+a\tau -s;\tau)} z^{-s} , \right] 
	\end{aligned}
\end{equation*}
that is, 
\begin{equation*}
	\begin{aligned}
		f(z) & =    \frac{G(\varphi+a\tau  ;\tau)}{G(\varphi+a\tau -s;\tau)} \sum_{n=0}^\infty \frac{G(\varphi+ n;\tau)}{G(\varphi+a\tau +n;\tau)} (-z)^n  , 
	\end{aligned} 
\end{equation*}
which is the representation of $\operatorname{E}_{a,m,l}(-z)$ as in \eqref{def.KS.G.pochhammer}. 

In conclusion, the Kilbas-Saigo function can be written as 
\begin{equation*}
	\operatorname{E}_{a,m,l}(-z) =  \frac{1}{2\pi i}\frac{G(\varphi+a\tau  ;\tau)}{G(\varphi;\tau)} \int_{C} \frac{\Gamma(s)\Gamma(1-s)G(\varphi  -s;\tau)}{G(\varphi+a\tau -s;\tau)} z^{-s} \, ds 
\end{equation*}
with  	$\operatorname{Re}z > 0 . $
\end{proof}

\begin{theorem} \label{KS_asymtptocis_complex_z}
For $z \in \mathbb{C}$ such that $\operatorname{Re}(z)>0$ the following asymptotic formula is valid:
\begin{equation*}
\operatorname{E}_{a,m,l}(-z)=\frac{\Gamma(1+a(l-m+1))}{\Gamma(1+a(l-m))}z^{-1}+\mathcal{O}(|z|^{-\delta}), \quad \text{ as } |z| \to \infty,
\end{equation*}
where $0 <a <1$, $m>0$, $l>m-1/a$ and $1<\delta \leq 3/2$.
\end{theorem}

\begin{proof}
For this matter we cannot use an arc on the right half-plane (as was done with an arc in the left half-plane, as illustrated in  Figure~\ref{fig.2}) because for this 
arc we would have $\cos\theta > 0$ and then 
the leading term in the expansion 
in \eqref{exp.log.integrand} would be    
$$
\alpha \cos\theta R \log{R} > 0 ,
$$
and consequently the integral along this arc would diverges for $R \to \infty$. 

To choose another contour we first need to look at the relationship between the poles 
located to the right of $C$. The first pole of $\Gamma(1-s)$ is $s=1$ and the first pole of
$G(\varphi+a\tau -s;\tau)$ is $s = \varphi + a\tau$. We will assume that 
\begin{equation*}
\label{choice.varphi}
\varphi > 1 . 
\end{equation*}
Recalling eq.\eqref{def.varphi}, this condition is equal to $l > m - 1/a$. 
Then it follows that the first pole of 
$1/G(\varphi+a\tau -s;\tau)$ is located to the right of the first pole of $G(1-s)$, that is,  
\begin{equation*}
	\varphi + a\tau > 1 . 
\end{equation*}
The second pole to the right of $C$ is therefore $\min(\varphi + a\tau,2)$. Let us denote by 
$\delta$ the midpoint between these two poles, that is, 
\begin{equation}
	\label{def.asympt.delta}
	\delta = \frac{1}{2}\left[1+\min(\varphi+a\tau,2)\right] . 
\end{equation}
Note that $\delta \in (1,3/2]$. 

As an alternative contour we will consider the closed contour $\mathcal{C} = C_0 \, \cup \, C_\uparrow \, \cup \, C^\prime \, \cup C_\downarrow$ illustrated in Figure~\ref{fig.3}. The residue theorem gives  
\begin{equation}
	\label{asymptotic.aux}
	\begin{split}
		\frac{1}{2\pi i}\frac{G(\varphi+a\tau ;\tau)}{G(\varphi ;\tau)} \int_{\mathcal{C}} \frac{\Gamma(s)\Gamma(1-s)G(\varphi  -s;\tau)}{G(\varphi+a\tau -s;\tau)} z^{-s} \, ds & \\[1ex]
		= 
		\frac{G[(1+a(l+1))\tau;\tau]}{G[(1+al)\tau;\tau]} \frac{G[(1+al)\tau-1;\tau]}{G[(1+a(l+1))\tau-1;\tau]} z^{-1} . &
	\end{split}
\end{equation}
and using \eqref{ap.B.general.G.tau} with $\varphi = (1+al)\tau$ and  $\tau = 1/(am)$, 
\begin{equation}
	\label{asymptotic.aux.2}
	\begin{split}
		& \frac{1}{2\pi i}\frac{G(\varphi+a\tau ;\tau)}{G(\varphi ;\tau)} \int_{\mathcal{C}} \frac{\Gamma(s)\Gamma(1-s)G(\varphi  -s;\tau)}{G(\varphi+a\tau -s;\tau)} z^{-s} \, ds   \\[1ex] 
		& \phantom{\frac{1}{2\pi i}\frac{G[(1+a(l+1))\tau;\tau]}{G[(1+al)\tau;\tau]}} = \frac{\Gamma(1+a(l-m+1))}{\Gamma(1+a(l-m))}  z^{-1} .   
	\end{split}
\end{equation}

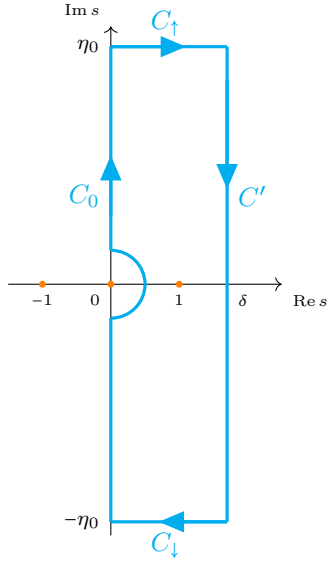
\begin{figure}[h]
	\begin{center}
		\begin{tikzpicture}[scale=0.9]
			\draw[->] (-1.5,0) -- (2.5,0) node[below right] {${\scriptscriptstyle \operatorname{Re}s}$};
			\draw[->] (0,-3.7) -- (0,3.8) node[above left] {${\scriptscriptstyle \operatorname{Im}s}$};
			\draw[orange,fill=orange] (0,0) circle (1.2pt);
			\draw[orange,fill=orange] (-1,0) circle (1.2pt);
			\draw[orange,fill=orange] (1,0) circle (1.2pt);
			\node[below] at (-1,0) {${\scriptscriptstyle -1}$};
			\node[below] at (1,0) {${\scriptscriptstyle 1}$};
			\node[below left] at (0,0) {${\scriptscriptstyle 0}$};
			\draw[cyan,very thick] (0,0.5) -- (0,3.5);
			\draw[cyan,very thick] (1.7,-3.5) -- (1.7,3.5);
			\draw[cyan,very thick] (0,-0.5) -- (0,-3.5);
			\draw[cyan,very thick] (0,-3.5) -- (1.7,-3.5);
			\draw[cyan,very thick] (0,3.5) -- (1.7,3.5); 
			\draw[->,>=triangle 45,cyan,very thick] (0,1) -- (0,1.9);
			\draw[->,>=triangle 45,cyan,very thick] (1.7,3) -- (1.7,1.4);
			\draw[->,>=triangle 45,cyan,very thick] (0,3.5) -- (1.1,3.5);
			\draw[->,>=triangle 45,cyan,very thick] (1.7,-3.5) -- (0.7,-3.5);
			\draw[cyan, very thick] (0,-0.5) arc (-90:90:0.5);
			\node[left] at (0,1.3) {\footnotesize ${\textcolor{cyan}{C_0}}$};
			\node[right] at (1.7,1.3) {\footnotesize ${\textcolor{cyan}{C^\prime}}$};
			\node[above] at (0.8,3.5) {\footnotesize $\textcolor{cyan}{C_\uparrow}$};
			\node[below] at (0.8,-3.5) {\footnotesize $\textcolor{cyan}{C_\downarrow}$};
			\node[below right] at (1.7,0) {$\scriptscriptstyle \delta$};
			\node[left] at (0,-3.5) {$\scriptstyle -\eta_0$};
			\node[left] at (0,3.5) {$\scriptstyle \eta_0$};
		\end{tikzpicture}
	\end{center}
	\caption{Integration contour for \eqref{asymptotic.aux}.\label{fig.3}}
\end{figure}

Let $\xi = R\cos\theta$ and $\eta = R \sin\theta$. 
Let us analyse the behaviour of the integrals along $C_\uparrow$ and $C_\downarrow$ 
as $\eta \to \pm \infty$. From \eqref{exp.log.integrand} we have for $\eta \to \pm \infty$ that 
\begin{equation*}
	\log{\mathcal{I}(s)} = \eta \operatorname{arg}z - \pi |\eta| - a \eta \arctan\frac{\eta}{\xi} + 
	a \pi |\eta| + \mathcal{O}(\log|\eta|) . 
\end{equation*}
Thus $\mathcal{I}(s)$ decays exponentially for $\eta \to \infty$ if 
\begin{equation*}
	\pm \arg{z} - \pi - a\frac{\pi}{2} + a\pi < 0 
\end{equation*}
with the plus sign for the case $\eta \to \infty$ and the minus sign for $\eta \to -\infty$. 
So the integrals along $C_\uparrow$ and $C_\downarrow$ vanishes for $\eta \to \pm \infty$ if 
\begin{equation*}
	|\arg{z}| < \left(1 - \frac{a}{2}\right)\pi 
\end{equation*}
or 
\begin{equation*}
	\operatorname{Re}z > 0  
\end{equation*}
for $0 < a < 1$.

Now we consider the integral along $C^\prime$ for $\eta \to \infty$. Let us denote
it by $\mathcal{J}(\delta)$, i.e.,
\begin{equation*}
	\label{int.J}
	\mathcal{J}(\delta) = -\frac{1}{2\pi i}\frac{G(\varphi+a\tau;\tau)}{G(\varphi;\tau)} 
	\int_{\delta-i\infty}^{\delta + i\infty} 
	\frac{\Gamma(s)\Gamma(1-s)G(\varphi  -s;\tau)}{G(\varphi+a\tau -s;\tau)} z^{-s} \, ds ,
\end{equation*}
where we recall that $\delta \in (1,3/2]$. 
Taking $s = \delta + i\eta$ and changing the integration variable to $\eta$ we can write 
\begin{equation*}
	\label{int.J.aux}
	|\mathcal{J}(\delta)| \leq \frac{|z|^{-\delta}}{2\pi} \frac{|G(\varphi+a\tau;\tau)|}{|G(\varphi;\tau)|} \int_{-\infty}^\infty 
	\frac{\pi}{|\sin{\pi(\delta+i\eta)}|}\frac{|G(\varphi-\delta-i\eta ;\tau)|}{|G(\varphi+a\tau-\delta-i\eta;\tau)|}d\eta ,
\end{equation*}
where we used Euler's reflection formula. Since 
\begin{equation*}
	|\sin{\pi(\delta+i\eta)}| \geq |\sin\pi\delta| \cosh{\pi\eta} 
\end{equation*}
we can write 
\begin{equation}
	\label{int.J.aux.2}
	|\mathcal{J}(\delta)| \leq \frac{|z|^{-\delta}}{2|\sin\pi\delta|} \frac{|G(\varphi+a\tau;\tau)|}{|G(\varphi;\tau)|} \int_{-\infty}^\infty 
	\frac{1}{ \cosh{\pi\eta} }\frac{|G(\varphi-\delta-i\eta ;\tau)|}{|G(\varphi + a\tau -\delta-i\eta;\tau)|} d\eta .
\end{equation}

Next we need an expression for the quotient of the double gamma functions in \eqref{int.J.aux.2}. 
Using \eqref{ap.B.G.tau.2} we obtain 
\begin{equation}
	\label{modulus.G}
	|G(x+iy;\tau)| = |G(x;\tau)| \operatorname{e}^{\left[-\frac{\tilde{b}(\tau)y^2}{2\tau^2}\right]} 
	\Theta_0(x,y) \prod_{m=1}^\infty \Theta_m(x,y) \operatorname{e}^{\left[-\frac{y^2}{2}\psi^\prime(m\tau)\right]} , 
\end{equation} 
where we defined 
\begin{equation*}
	\Theta_m(x,y) = \sqrt{\prod_{k=0}^\infty \left[1+\frac{y^2}{(x+m\tau+k)^2}\right]} ,
\end{equation*}
and we have used 
\begin{equation*}
	|\Gamma(x+iy)| = \frac{|\Gamma(x)|}{\Theta_0(x,y)} , 
\end{equation*}
which follows from Weierstrass' definition of the gamma function. 
Using \eqref{modulus.G} we have 
\begin{equation*}
	\begin{split}
		\frac{|G(\varphi-\delta-i\eta ;\tau)|}{|G(\varphi + a\tau -\delta-i\eta;\tau)|} = &  
		\frac{|G(\varphi-\delta;\tau)|}{|G(\varphi + a\tau -\delta;\tau)|} \\[1ex]
		& \cdot \prod_{m=0}^\infty \frac{\Theta_m(\varphi-\delta,\eta)}{\Theta_m(\varphi + a\tau-\delta,\eta)} .
	\end{split}
\end{equation*}
Recalling that $a < 1$ we have 
\begin{equation*}
	\begin{split}
		\Theta_m(\varphi + a\tau -\delta,\eta) & = \sqrt{\prod_{k=0}^\infty \left[1 + \frac{\eta^2}{(\varphi - \delta + a\tau + m \tau + k)^2}\right]} \\[1ex]
		& \geq \sqrt{\prod_{k=0}^\infty \left[1 + \frac{\eta^2}{(\varphi - \delta + (m+1) \tau + k)^2}\right]} \\[1ex]
		& = \Theta_{m+1}(\varphi-\delta,\eta) 
	\end{split}
\end{equation*}
So we can write 
\begin{equation*}
	\begin{split}
		\prod_{m=0}^\infty \frac{\Theta_m(\varphi-\delta,\eta)}{\Theta_m(\varphi + a\tau-\delta,\eta)}  & = 
		\lim_{N\to \infty} \frac{\Theta_0(\varphi-\delta,\eta)\Theta_1(\varphi-\delta,\eta)\cdots \Theta_N(\varphi-\delta,\eta)}{\Theta_0(\varphi + a\tau-\delta,\eta)\Theta_1(\varphi + a\tau-\delta,\eta)\cdots \Theta_N(\varphi + a\tau-\delta,\eta)} \\[1ex]
		& \leq \lim_{N\to \infty} \frac{\Theta_0(\varphi-\delta;\eta)}{\Theta_N(\varphi-\delta;\eta)} = \Theta_0(\varphi-\delta;\eta) . 
	\end{split}
\end{equation*}
Thus we have 
\begin{equation*}
	\frac{|G(\varphi-\delta-i\eta ;\tau)|}{|G(\varphi + a\tau -\delta-i\eta;\tau)|} \leq  
	\frac{|G(\varphi-\delta;\tau)|}{|G(\varphi + a\tau -\delta;\tau)|}\Theta_0(\varphi-\delta,\eta) .
\end{equation*}

Returning to \eqref{int.J.aux.2}, we have 
\begin{equation}
	\label{int.J.aux.3}
	|\mathcal{J}(\delta)| \leq \frac{|z|^{-\delta}}{2|\sin\pi\delta|} \frac{|G(\varphi+a\tau;\tau)|}{|G(\varphi;\tau)|}\frac{|G(\varphi-\delta;\tau)|}{|G(\varphi + a\tau -\delta;\tau)|}  \int_{-\infty}^\infty 
	\frac{1}{ \cosh{\pi\eta} }\Theta_0(\varphi-\delta,\eta)  d\eta .
\end{equation}
Since $\delta \leq 3/2$ and $\varphi > 0$ (recall that $l > -1/a$ and $\tau > 0$), we have 
\begin{equation*}
	\varphi - \delta \geq -\frac{3}{2} . 
\end{equation*}
So we have 
\begin{equation*}
	\Theta_0(\varphi-\delta,\eta) \leq 
	\sqrt{\prod_{k=0}^\infty \left[1+\frac{\eta^2}{(-3/2+k)^2}\right]} = 
	\sqrt{\left(1+4\eta^2/9\right)\left(1+4\eta^2\right) 
		\cosh{\pi \eta}} ,
\end{equation*}
where we have used the Euler's hyperbolic product formula \cite{Salwinski} 
\begin{equation*}
	\prod_{k=0}^\infty \left[1+\frac{\eta^2}{(1/2+k)^2}\right] = 
	\cosh{\pi \eta} . 
\end{equation*}
So from \eqref{int.J.aux.3} we have 
\begin{equation*}
	|\mathcal{J}(\delta)| \leq K_\delta |z|^{-\delta} , 
\end{equation*}
where 
\begin{equation*}
	\label{int.J.aux.4}
	K_\delta = \frac{1.96976}{|\sin\pi\delta|} \frac{|G(\varphi+a\tau;\tau)|}{|G(\varphi;\tau)|}\frac{|G(\varphi-\delta;\tau)|}{|G(\varphi + a\tau -\delta;\tau)|}  
\end{equation*}
where the improper integral has been numerically evaluated using Mathematica 14, 
\begin{equation*}
	\int_{-\infty}^\infty 
	\sqrt{\frac{\left(1+ 4\eta^2/9\right)\left(1+4\eta^2\right)}{
			\cosh{\pi \eta}} } d\eta = 3.93953.  
\end{equation*}

Finally from \eqref{asymptotic.aux.2} we conclude that 
\begin{equation*}
	\operatorname{E}_{a,m,l}(-z) =  \frac{\Gamma(1+a(l-m+1))}{\Gamma(1+a(l-m))}  z^{-1} + \mathcal{O}(|z|^{-\delta}) 
\end{equation*}
for $\operatorname{Re}z > 0$ and with $\delta \in (1,3/2]$ given by \eqref{def.asympt.delta}. 

\end{proof}

Therefore, asymptotic formula \eqref{KS_realz_asympt} is still valid for complex numbers, as long as its real part is strictly positive.

In next two subsections we explore the connection between newly defined fractional derivative and KS function. In fact, we show that KS function is eigenfunction for this operator.

\subsection{A first-order equation involving $\boldsymbol{\mathcal{D}^{(\alpha,\gamma)}_t}$ and the KS function } \label{first_order_KS}

Consider the fractional differential equation 
\begin{equation}
\label{fde.1}
\operatorname{\mathcal{D}}^{(\alpha,\gamma)}_t f(t) + \kappa f(t) = 0 . 
\end{equation}
This equation can be used as a model for anomalous relaxation \cite{Capelas1,Capelas2}. 
 We will look for solutions of \eqref{fde.1} of the form
\begin{equation}
\label{eq.3}
f(t) = \sum_{n=0}^\infty f_n t^{(\alpha+\gamma)n} . 
\end{equation}
Let us introduce the compact notation 
\begin{equation*}
[ \mathsf{n}]^\beta_\alpha = \frac{\Gamma(\beta n + 1)}{\Gamma(\beta n - \alpha + 1)} 
\end{equation*}
with
\begin{equation*}
[ {\mathsf{n}_1 \times  \mathsf{n}_2 \times  \cdots \times  \mathsf{n}_k} ]^\beta_\alpha = 
[ \mathsf{n}_1 ]^\beta_\alpha[ \mathsf{n}_2 ]^\beta_\alpha\cdots 
[ \mathsf{n}_k ]^\beta_\alpha
\end{equation*}
and 
\begin{equation*}
[\mathsf{n!}]^\beta_\alpha = [\mathsf{n}\times (\mathsf{n-1})\times \cdots \times \mathsf{1}]^\beta_\alpha .
\end{equation*}
Thus we have 
\begin{equation}
\label{der.order.1}
\left(\operatorname{\mathcal{D}}^{(\alpha,\gamma)}_t\right) f(t) = \sum_{n=0}^\infty 
f_{n+1}[ \mathsf{n+1}]^{\alpha+\gamma}_\alpha t^{(\alpha+\gamma)n} 
\end{equation}
and from \eqref{fde.1}, 
\begin{equation*}
f_{n+1}[ \mathsf{n+1}]^{\alpha+\gamma}_\alpha + \kappa f_n = 0 , \qquad n=0,1,2,\ldots 
\end{equation*}
Thus the solution is 
\begin{equation*}
\label{sol.fde1}
f(t) = \sum_{n=0}^\infty \frac{(-\kappa t^{\alpha+\gamma})^n}{[\mathsf{n!}]^{\alpha+\gamma}_\alpha} f_0 ,
\end{equation*}
where we defined $[\mathsf{0!}]^{\alpha+\gamma}_\alpha = 1$. This corresponds to a KS function with 
\begin{equation*}
a = \alpha, \qquad m = 1 + \frac{\gamma}{\alpha} , \qquad l = \frac{\gamma}{\alpha} ,  
\end{equation*}
that is, 
\begin{equation} \label{first_order_KS_solution}
f(t) = f_0 \, \operatorname{E}_{\alpha,1+\gamma/\alpha,\gamma/\alpha}\left(-\kappa t^{\alpha+\gamma}\right),
\end{equation}
where $f_0$ is arbitrary. For simplicity we will assume $f_0=1$.
\begin{remark} When $\gamma = 0$ we have $\mathcal{D}^{(\alpha,\gamma)}_t = 
{\sideset{_{\scriptscriptstyle C}^{}}{_t^{(\alpha)}}{\operatorname{\mathcal{D}}}}$ and 
\begin{equation*}
[\mathsf{n}]^{\alpha+\gamma}_\alpha = [\mathsf{n}]_\alpha^\alpha = 
\frac{\Gamma(\alpha n+1)}{\Gamma(\alpha(n-1)+1)} , 
\end{equation*}
so that 
\begin{equation*}
[\mathsf{n!}]^{\alpha}_\alpha = \frac{\Gamma(\alpha n + 1)}{\Gamma(1)} . 
\end{equation*}
Thus 
\begin{equation*}
\operatorname{E}_{\alpha,1,0}\left(-\kappa t^{\alpha}\right) = 
\sum_{n=0}^\infty \frac{(-\kappa t^{\alpha})^n}{\Gamma(\alpha n + 1)} = \operatorname{E}_\alpha(-\kappa 
t^\alpha) ,
\end{equation*}
where $\operatorname{E}_\alpha(\cdot)$ is the Mittag-Leffler function (see \cite{MittagLeffler}). 
\end{remark}
\begin{remark}
These observations show that KS function $\operatorname{E}_{\alpha,1+\gamma/\alpha,\gamma/\alpha}\left(-\kappa t^{\alpha+\gamma}\right)$ is an eigenfunction for the non-local differential operator $\operatorname{\mathcal{D}}^{(\alpha,\gamma)}_t$. This is in accordance with the fact that Mittag-Leffler function $\operatorname{E}_\alpha(-\kappa t^{\alpha})$ is an eigenfunction for Caputo fractional derivative $ {\sideset{_{\scriptscriptstyle C}^{}}{_t^{(\alpha)}}{\operatorname{\mathcal{D}}}}$  and exponential function $\exp(-\kappa t)$ is eigenfunction for standard first-order differential operator $\frac{d}{dt}$.
\end{remark}
Therefore, KS function and the operator $\operatorname{\mathcal{D}}^{(\alpha,\gamma)}_t$ provide the right tool for generalising results regarding fractional processes involving Caputo fractional derivative.

In fact, results from \cite{Pearson} can be fully extended with this approach, by replacing Caputo fractional derivative with the proposed operator $\operatorname{\mathcal{D}}^{(\alpha,\gamma)}_t$. This leads to more general fractional Pearson diffusions with more general time change model. We will refer to this new model as stretched non-local Pearson diffusion. By this we emphasize the role of the new stretching parameter $\gamma$. By setting $\gamma=0$ we recover results from \cite{Pearson}.
\bigskip

\subsection{A second-order equation involving $\boldsymbol{\mathcal{D}^{(\alpha,\gamma)}_t}$ and the KS function } 
Consider the fractional differential equation 
\begin{equation} \label{eq.2}
\left(\operatorname{\mathcal{D}}^{(\alpha,\gamma)}_t\right)^2 f(t) + a 
\left(\operatorname{\mathcal{D}}^{(\alpha,\gamma)}_t\right)f(t) + 
b f(t) = 0 , 
\end{equation}
where $a$ and $b$ are constants. 
The case $\gamma = 0$ and $\alpha = 1$ is related to the usual time part of 
telegraph equation (for details see \cite{Telegraph_eq}). 

We will seek solutions to equation \eqref{eq.2} using the form described in equation \eqref{eq.3} 
The term $\left(\operatorname{\mathcal{D}}^{(\alpha,\gamma)}_t\right)f(t)$ 
is given by \eqref{der.order.1} and 
 \begin{equation*}
\left(\operatorname{\mathcal{D}}^{(\alpha,\gamma)}_t\right)^2 f(t) = \sum_{n=0}^\infty 
f_{n+2}[   (\mathsf{n+2}) \times (\mathsf{n+1}) ]^{\alpha+\gamma}_\alpha t^{(\alpha+\gamma)n} ,
\end{equation*}
which in \eqref{eq.2} gives the recurrence relation 
\begin{equation*}
f_{n+2} = \frac{(-a)}{[\mathsf{n+2}]^{\alpha+\gamma}_\alpha} f_{n+1} + 
\frac{(-b)}{[   (\mathsf{n+2})\times (\mathsf{n+1}) ]^{\alpha+\gamma}_\alpha} f_{n} . 
\end{equation*}
Using this recurrence relation we obtain 
\begin{equation}
\label{gen.coeff}
f_{n} = \frac{\mathcal{U}_{n}(-a,-b)}{[ \mathsf{n!}]^{\alpha+\gamma}_\alpha} [ \mathsf{1}]^{\alpha+\gamma}_\alpha f_1 + 
\frac{(-b)\mathcal{U}_{n-1}(-a,-b)}{[ \mathsf{n}!]^{\alpha+\gamma}_\alpha}  f_0 , \qquad n = 2,3,\ldots
\end{equation}
where $\mathcal{U}_{n}(-a,-b)$  is given by 
\begin{equation*}
\mathcal{U}_{n}(-a,-b) = \sum_{j=0}^{\lfloor (n-1)/2\rfloor} \binom{n-1-j}{j} (-a)^{n-1-2j}
(-b)^j,  
\end{equation*}
and $\lfloor \cdot \rfloor$ denotes the floor function and $n = 1,2,3,\ldots$. 
Note that \eqref{gen.coeff} also holds for $n=1$ if we define $\mathcal{U}_0(-a,-b) = 0$. 

According to \cite{Fibonacci}, bivariate Fibonacci polynomial $u_n(x,y)$ can be expressed as
\begin{equation}\label{BivFibFormula}
\mathcal{U}_n(-a,-b)=\frac{(a^{*})^n-(b^{*})^n}{a^{*}-b^{*}},
\end{equation}
where 
\begin{equation} \label{a*b*}
a^{*}=\frac{-a}{2}+\sqrt{\left(\frac{-a}{2}\right)^2-b}, \quad  b^{*}=\frac{-a}{2}-\sqrt{\left(\frac{-a}{2}\right)^2-b}.
\end{equation}

Notice if $b>(\frac{-a}{2})^2$ then $a^{*}$ and $b^{*}$ are complex conjugate numbers so that $\mathcal{U}_n(-a,-b)$ is still real number for all pairs $a$ and $b$.
Now, from \eqref{gen.coeff} we have
\begin{equation*}
f_{n} = \frac{-a\cdot \mathcal{U}_{n}(-a,-b)-b\cdot \mathcal{U}_{n-1}(-a,-b)}{[ \mathsf{n!}]^{\alpha+\gamma}_\alpha}=\frac{\mathcal{U}_{n+2}(-a,-b)}{[ \mathsf{n!}]^{\alpha+\gamma}_\alpha} , \qquad n = 2,3,\ldots
\end{equation*}
Taking this into account together with formula \eqref{BivFibFormula} for general solution we have

\begin{equation*}
f(t)=f_0+\frac{[1]_{\alpha}^{\alpha+\gamma}a^* f_1-b f_0}{a^*(a^*-b^*)}\sum_{n=1}^{\infty}\frac{\left(a^{*} t^{\alpha+\gamma}\right)^n}{[n!]_{\alpha}^{\alpha+\gamma}}+\frac{b f_0-[1]_{\alpha}^{\alpha+\gamma}b^* f_1}{b^* (a^*-b^*)}\sum_{n=1}^{\infty}\frac{\left(b^* t^{\alpha+\gamma}\right)^n}{[n!]_{\alpha}^{\alpha+\gamma}}.
\end{equation*}
Let
\begin{equation}\label{K1K2}
	K_1=\frac{[1]_{\alpha}^{\alpha+\gamma}a^* f_1-b f_0}{a^*(a^*-b^*)}, \quad K_2=\frac{b f_0-[1]_{\alpha}^{\alpha+\gamma}b^* f_1}{b^* (a^*-b^*)}.
\end{equation}
Notice that $K_1+K_2=f_0$.
Therefore
\begin{equation*}
	f(t)=K_1\sum_{n=0}^{\infty}\frac{\left(a^{*} t^{\alpha+\gamma}\right)^n}{[n!]_{\alpha}^{\alpha+\gamma}}+K_2\sum_{n=0}^{\infty}\frac{\left(b^* t^{\alpha+\gamma}\right)^n}{[n!]_{\alpha}^{\alpha+\gamma}} .
\end{equation*}
Finally, taking into account definition of KS function we obtain solution in the form
\begin{equation*}
	f(t)=K_1\operatorname{E}_{\alpha,1+\gamma/\alpha,\gamma/\alpha}\left(a^*t^{\alpha+\gamma}\right)+K_2\operatorname{E}_{\alpha,1+\gamma/\alpha,\gamma/\alpha}\left(b^*t^{\alpha+\gamma}\right),
\end{equation*}
where $f_0$ and $f_1$ are arbitrary. For the rest of the paper we assume $f_0=1$ and $f_1=\frac{1}{[1]_{\alpha}^{\alpha+\gamma}}$, which simplifies constants $K_1$ and $K_2$ to
\begin{equation*}
	K_1=\frac{a^*-b }{a^*(a^*-b^*)}, \quad K_2=\frac{b -b^*}{b^* (a^*-b^*)}
\end{equation*}
with $K_1+K_2=1$.

\begin{remark}
Taking into account that $a^*$ and $b^*$ are complex conjugate numbers, it is not hard to see that $K_1$ and $K_2$ are also conjugates, which implies 
\begin{align*}
f(t)&=K_1\operatorname{E}_{\alpha,1+\gamma/\alpha,\gamma/\alpha}\left(a^*t^{\alpha+\gamma}\right)+K_2\operatorname{E}_{\alpha,1+\gamma/\alpha,\gamma/\alpha}\left(b^*t^{\alpha+\gamma}\right) \\
    &=K_1\operatorname{E}_{\alpha,1+\gamma/\alpha,\gamma/\alpha}\left(a^*t^{\alpha+\gamma}\right)+\overline{K_1\operatorname{E}_{\alpha,1+\gamma/\alpha,\gamma/\alpha}\left(a^*t^{\alpha+\gamma}\right)} \\
    &=2Re\left(K_1\operatorname{E}_{\alpha,1+\gamma/\alpha,\gamma/\alpha}\left(a^*t^{\alpha+\gamma}\right)\right).
\end{align*}
Therefore, the solution of \eqref{eq.2} is a real-valued function.
\end{remark}

\begin{remark}
When $\gamma = 0$  we have
\begin{equation*}
	f(t) = K_1 \, \operatorname{E}_{\alpha}\left(a^* t^{\alpha}\right) + 
	K_2 \, \operatorname{E}_{\alpha}\left(b^*t^{\alpha}\right) ,
\end{equation*}
where $\operatorname{E}_{\alpha}(\cdot)$ is the Mittag-Leffler function. 
\end{remark}

\section{Pearson diffusions and their classification}  \label{PearsonDiffusions}
 
For completeness, we will briefly summarize the discussion of \cite{Pearson} concerning 
the classification of Pearson diffusions. The stationary density of a diffusion process $\boldsymbol{m}(x)$ (if exists) satisfies the time-independent Fokker-Planck equation, which for Pearson diffusions reduces to 
\begin{equation*}
	\frac{\boldsymbol{m}^\prime(x)}{\boldsymbol{m}(x)} = \frac{\mu(x) - D^\prime(x)}{D(x)} = \frac{(a_0-d_1) + (a_1-2d_2)x}{d_0+d_1 x + d_2 x^2} . 
\end{equation*}  
Note that this is the equation for the weight functions of the hypergeometric equation \cite{Nikiforov}
\begin{equation}
	\label{hg.eq}
	D(x) g^{\prime\prime}(x) + \mu(x)g^\prime(x) + \lambda g(x) = 0 . 
\end{equation}

There are six different types of Pearson diffusions corresponding to six different classes of solutions of the equation (\ref{hg.eq}). These classes depend on the degree of $D(x)$, which can be zero, one, or two, and if $D(x)$ has a degree of two, on the sign of the discriminant of $D(x$). When $D(x)$ is constant, linear, or quadratic with a positive discriminant, the spectrum of $\mathcal{G}$ is purely discrete, and the eigenfunctions are the well-known classical orthogonal polynomials, and the eigenvalues are of the form
\begin{equation}
	\label{eq.eigenvalues}
	\lambda_n = -n[a_1 + d_2(n-1)] , \qquad n = 0,1,2,\ldots. 
\end{equation}
For the remaining cases, we have a mixed spectrum. We will restrict ourselves to the purely discrete case, which can be classified as follows \cite{Pearson,Meerschaert}:

\bigskip 
\noindent (i) The \textit{Ornstein-Uhlenbeck (OU) process} corresponds to a constant diffusion coefficient. Denoting $D(x) = \theta \sigma^2$ and $\mu(x) = -\theta(x-\mu)$, 
where $\mu \in \mathbb{R}$, $\sigma^2>0$ and $\theta > 0$ is a correlation function parameter, we obtain 
\begin{equation*}
	\boldsymbol{m}(x) = \frac{1}{\sqrt{2\pi \sigma^2}}\operatorname{exp}\left(
	-\frac{(x-\mu)^2}{2\sigma^2}\right) , 
\end{equation*}
with $x\in \mathbb{R}$, i.e.~stationary distribution is normal with parameters $\mu$ and $\sigma^2$. 
The eigenvalues of \eqref{eigen.hg} 
are $\lambda_n = \theta n$ ($n = 0,1,2,\ldots$) and
the corresponding eigenfunctions are the Hermite polynomials $H_n(x)$, which can be 
defined in terms of the Rodrigues formula 
\begin{equation*}
	H_n(x) =  \frac{(-1)^n}{\boldsymbol{m}(x)}\frac{d^n\;}{dx^n}\boldsymbol{m}(x) , \quad 
	n = 0,1,2,\ldots . 
\end{equation*}

\bigskip

\noindent (ii) The \textit{Cox-Ingersoll-Ross (CIR) process} corresponds to $D(x) = d_0 + d_1 x$. The usual parametrization of the CIR process is $D(x) = \theta x/a$ and 
$\mu(x) = -\theta(x-b/a)$ with $\theta, a, b >0$. In this case 
\begin{equation*}
	\boldsymbol{m}(x) = \frac{a^b}{\Gamma(b)} x^{b-1}\operatorname{e}^{-ax} ,
\end{equation*}
with $x>0$, which corresponds to gamma stationary distribution with parameters $b$ and $a$.  
The eigenvalues of \eqref{eigen.hg} are $\lambda_n = \theta n$ ($n = 0,1,2,\ldots$)  
the corresponding eigenfunctions are the Laguerre polynomials $L^\nu_n(a x)$ ($\nu = b-1$), which can be 
defined in terms of the Rodrigues formula 
\begin{equation*}
	L^\nu_n(x) = \frac{x^{-\nu}\operatorname{e}^x}{n!} \frac{d^n\;}{dx^n}\left( 
	x^{\nu+n}\operatorname{e}^{-x}\right) , \quad 
	n = 0,1,2,\ldots . 
\end{equation*}

\bigskip

\noindent (iii) The \textit{Jacobi process} corresponds to the case when $D(x)$ 
is a second-order polynomial with positive discriminant. Thus $D(x)$ can be 
written as $D(x) = d_2 (x-x_1)(x-x_2)$ with $x_1 \neq x_2$. The usual 
parametrization is $D(x) = 1-x^2$ and $\mu(x) = -(a+b+2)x+b-a$ with $a,b>-1$. We have
\begin{equation*}
	\boldsymbol{m}(x) = \frac{(1-x)^a(1+x)^b}{B(a+1,b+1) 2^{a+b+1}} , 
\end{equation*}
for $x\in (-1,1)$ and $B(\cdot,\cdot)$ is the beta function. This is beta stationary distribution with parameters $a$ and $b$. The eigenvalues 
are $\lambda_n = n(n+a+b+1)$ ($n=0,1,2,\ldots$) and the corresponding 
eigenfunctions are the Jacobi polynomials $P_n^{(a,b)}(x)$ given by 
\begin{equation*}
	P_n^{(a,b)}(x) = \frac{(-1)^n (1-x)^{-a}(1+x)^{-b}}{2^n n!} \frac{d^n\;}{dx^n} 
	\left[(1-x)^{a+n}(1+x)^{b+n}\right] . 
\end{equation*}
Let us observe the Cauchy problem  \eqref{kolmo} with solution of the form
$g(y,t) = T(t) \varphi(y)$. Then we have 
\begin{equation}
\label{eigen.hg}
\mathcal{G}[\varphi] = -\lambda \varphi ,
\end{equation}  
which is the hypergeometric equation (\ref{hg.eq}), 
where $\lambda$ was introduced as a separation constant. It is well known that there are
polynomial solutions $Q_n(x)$ of the hypergeometric equation corresponding to eigenvalues $\lambda_n$, and that they form an orthonormal system with weight function $\boldsymbol{m}(x)$, i.e. 
\begin{equation}
\label{ortho.eq}
\int Q_n(x) Q_m(x)\boldsymbol{m}(x)\, dx = \delta_{mn} c_n^2 . 
\end{equation}
If $T_n(t)$ is the corresponding time solution, we can write the general solution as 
\begin{equation}
\label{eq.g.aux.1}
g(y,t) = \sum_{n=0}^\infty a_n T_n(t)Q_n(y) , 
\end{equation}
where $a_n$ are arbitrary constants. Using the initial condition $g(y,0) = g(y)$ we have
\begin{equation*}
g(y) = \sum_{n=0}^\infty a_n T_n(0) Q_n(y) , 
\end{equation*}
and using \eqref{ortho.eq}, 
\begin{equation*}
a_n T_n(0) = \frac{1}{c_n^2} \int Q_n(x) g(x) \boldsymbol{m}(x)\, dx. 
\end{equation*}
Using this in \eqref{eq.g.aux.1} and recalling \eqref{eq.g.def} we can write 
the transition density for the Pearson diffusion as 
\begin{equation}\label{denpear}
\boldsymbol{p}(x,t;y) = \boldsymbol{m}(x) \sum_{n=0}^\infty \frac{T_n(t)}{T_n(0)c_n^2} 
Q_n(x)Q_n(y) . 
\end{equation}
Note that for \eqref{kolmo} we have $T_n(t) = \operatorname{e}^{-\lambda_n t}$ (for details on spectral representation of Pearson diffusion with discrete spectrum see e.g. \cite{Karlin}).

\begin{remark}
Normalizing constants $c_n$ for orthonormal polynomials in the previous discussion are
\begin{itemize}
\item $c_n=\frac{\sigma^n}{\sqrt{n!}}$ for Hermite polynomial $H_n(x)$, 
\item $c_n=\frac{1}{\sqrt{B(b, n)}}$ for Laguerre polynomial $L^{b-1}_n(x)$,
\item $c_n=\sqrt{\frac{2^{a+b+1}}{2n+a+b+1}\frac{\Gamma(n+a+1)\Gamma(n+b+1)}{\Gamma(n+1)\Gamma(n+a+b+1)}}$ for Jacobi polynomial $P^{(a,b)}_n(x)$ .
\end{itemize}
In the rest of the paper we assume we are using normalized orthonormal polynomials.  Instead of writing $Q_n(x)/c_n$ we will simply write $Q_n(x)$ for corresponding normalized polynomial.
\end{remark}

\section{Stretched non-local Pearson diffusions} \label{sec_fpd}
In this section we provide generalization of analytical results obtained in \cite{Pearson}.  Instead of Caputo fractional derivative of order $0 <\alpha <1$ we use its modification \eqref{Caputo_variation} with positive parameters $\alpha$ and $\gamma$ such that $0<\alpha+\gamma \leq 1$.

Our model for a new type of non-local fractional  Pearson diffusion is based on the fractional Cauchy problem
\begin{equation}\label{fpd_frac_eq}
\operatorname{\mathcal{D}}^{(\alpha,\gamma)}_t g(y,t)  = 
	\mathcal{G}[g], \quad g(y,0)=g(y).
\end{equation}
Taking spectral separation of variables approach as in \cite{Pearson} we seek solution $g(y,t)=Y(y) T(t)$ of \eqref{fpd_frac_eq}, where $T(t)$ is solution of equation
\begin{equation*}
\operatorname{\mathcal{D}}^{(\alpha,\gamma)}_t T(t)= -\lambda T(t)
\end{equation*}
and  $Y(y)$ is solution of equation
\begin{equation*}
	\mathcal{G}Y(y)= -\lambda Y(y).
\end{equation*}

Taking into account \eqref{first_order_KS_solution} (see Section \ref{first_order_KS})) 
we have temporal solution $T(t)=\operatorname{E}_{\alpha,1+\gamma/\alpha,\gamma/\alpha}\left(-\lambda t^{\alpha+\gamma}\right)$, while solution of space part equation is provided by corresponding eigenfunctions of Pearson diffusion (see Section \ref{PearsonDiffusions}).

Therefore, the only difference from \cite{Pearson} is in temporal part, where we have solution $T(t)=\operatorname{E}_{\alpha,1+\gamma/\alpha,\gamma/\alpha}\left(-\lambda t^{\alpha+\gamma}\right)$  (Kilbas-Saigo function) instead of $T(t)=\operatorname{E}_{\alpha}\left(-\lambda t^{\alpha}\right)$ (Mittag-Leffler function).
Furthermore, KS function is generalization of Mittag-Leffler function since for $\gamma=0$ we have $\operatorname{E}_{\alpha,1,0}\left(-\lambda t^{\alpha}\right)=\operatorname{E}_{\alpha}\left(-\lambda t^{\alpha}\right)$.
Finally, we have proved that KS function has analogous bounds and asymptotic behaviour as Mittag-Leffer function (see \eqref{KS_function_bounds}, \eqref{KS_realz_asympt},
	\eqref{KS_realz_asympt_spec}), therefore we can recover all analytical results as in \cite{Pearson}, which we state as follows.

\begin{theorem} \label{trans_dens_steched_Pearson}
	For the three classes of Pearson diffusions (OU, CIR and Jacobi) whose stationary density $\boldsymbol{m}$ and orthogonal polynomials $\left(Q_n, \, n \in \mathbb{N}_0\right)$ are given in Section \ref{PearsonDiffusions}, for any $0<\alpha+\gamma\leq1$, the series
	\begin{equation} \label{trans_fpd}
		p_{\alpha, \gamma}(x,t;y)=\boldsymbol{m}(x)\sum_{n=0}^{\infty}\operatorname{E}_{\alpha,1+\gamma/\alpha,\gamma/\alpha}\left(-\lambda_n t^{\alpha+\gamma}\right) Q_n(x)Q_n(y)
	\end{equation}
 converges for fixed $t>0$ and $x$, $y$ in the state space of the corresponding diffusion (OU, CIR or Jacobi).
\end{theorem}
\begin{proof}
The proof mimics the method in [\cite{Pearson}, Lemma 3.1.] with Mittag-Leffler function being replaced with Kilbas-Saigo function. Since Kilbas-Saigo function is extension of Mittag-Leffler function and has analogous bounds and asymptotics (see \eqref{KS_function_bounds},\eqref{KS_realz_asympt},
\eqref{KS_realz_asympt_spec}) we can just repeat the steps in [\cite{Pearson}, Lemma 3.1.] and therefore the proof is finished.
\end{proof}
\begin{theorem}\label{Cauchy_problem_Gen_streched}
	Let $h \in L_2(\boldsymbol{m}(x)dx)$ be such a function that $\sum_{n} a_n Q_n$, where $$a_n=\int_{S} h(x) Q_n(x)\boldsymbol{m}(x) dx,$$ converges to function $h$ uniformly on finite intervals (subsets of state space $S$ of corresponding diffusion). Then the fractional Cauchy problem
	\begin{equation} \label{Cauchy_problem_fpd}
		 \mathcal{D}^{(\alpha,\gamma)}_t g(t,y)= \mathcal{G} g(t,y), \quad g(0,y)=h(y) 
	\end{equation}
	has a strong solution given by
	\begin{equation} \label{Cauchy_problem_solution_fpd}
		g_{\alpha, \gamma}(t,y)=\sum_{n=0}^{\infty} \operatorname{E}_{\alpha,1+\gamma/\alpha,\gamma/\alpha}\left(-\lambda_n t^{\alpha+\gamma}\right) Q_n(y) a_n.
	\end{equation}
	The series in \eqref{Cauchy_problem_solution_fpd} converges absolutely for each  fixed $t>0$, $y \in S$, and \eqref{Cauchy_problem_fpd} holds pointwise.
\end{theorem}
\begin{proof}
	The proof is similar to that in [\cite{Pearson}, Theorem 3.2.] with two major differences. Firstly, we are using Caputo-type fractional derivative $\mathcal{D}^{(\alpha,\gamma)}_t$ instead of Caputo fractional derivative and secondly, analytical solution has Kilbas-Saigo function in place of Mittag-Leffler function. Recall that Kilbas-Saigo function is an eigenfunction of $\mathcal{D}^{(\alpha,\gamma)}_t$ just like Mittag-Leffler function is an eigenfunction for Caputo fractional derivative. Since the two functions have analogous bounds and asymptotics (see \eqref{KS_function_bounds},\eqref{KS_realz_asympt},
	\eqref{KS_realz_asympt_spec}) we can just repeat the steps in [\cite{Pearson}, Theorem 3.2.] and therefore the proof is finished.
\end{proof}
\begin{theorem}\label{Cauchy_problem_FP_streched}
	Let $h/\boldsymbol{m}(x) \in L_2(\boldsymbol{m}(x)dx)$ be such a function that $\sum_{n} a_n Q_n$, where $$a_n=\int_{S} h(y) Q_n(y) dy,$$ converges to function $h/\boldsymbol{m}$ uniformly on finite intervals (subsets of state space $S$ of corresponding diffusion). Then the fractional Cauchy problem
	\begin{equation} \label{Cauchy_problem2_fpd}
		 \mathcal{D}^{(\alpha,\gamma)}_t g(t,x)= \mathcal{L} g(t,x), \quad g(0,x)=h(x) 
	\end{equation}
	has a strong solution given by
	\begin{equation} \label{Cauchy_problem_solution2_fpd}
		g_{\alpha, \gamma}(t,x)=\boldsymbol{m}(x)\sum_{n=0}^{\infty} \operatorname{E}_{\alpha,1+\gamma/\alpha,\gamma/\alpha}\left(-\lambda_n t^{\alpha+\gamma}\right)Q_n(x) a_n.
	\end{equation}
	The series in \eqref{Cauchy_problem_solution2_fpd} converges absolutely for each  fixed $t>0$, $x \in S$, and \eqref{Cauchy_problem2_fpd} holds pointwise.
\end{theorem}
\begin{proof}
	The proof is similar to that in [\cite{Pearson}, Theorem 3.3.] and the reasoning is the same as in the last theorem, therefore we skip it.
\end{proof}
\section{Fractional hyperbolic Pearson diffusion}\label{sec_fhpd}
In this section we provide model for a fractional hyperbolic Pearson diffusion which is based on the following fractional Cauchy problem with stretching:
\begin{equation}
\label{frac.hyperbolic.pearson}
A \left(\operatorname{\mathcal{D}}^{(\alpha,\gamma)}_t\right)^2 g(y,t)  + B 
\left(\operatorname{\mathcal{D}}^{(\alpha,\gamma)}_t\right) g(y,t)  = 
\mathcal{G}[g] , \quad g(y,0)=g(y),
\end{equation}
where we assume $A\geq 0$ and $B\geq 0$. Note that \eqref{frac.kolmo} corresponds to the particular case  of \eqref{frac.hyperbolic.pearson} with $A = 0$ and $\gamma = 0$. 
This model combines elements of telegraph equation (in time)  and Pearson diffusion equation (in space). We will assume $0<\alpha+\gamma \leq1$.

Let us look for solutions of \eqref{frac.hyperbolic.pearson}. Writing $g(y,t) = Y(y) T(t)$, we obtain 
\begin{equation}
\label{eq.T}
 A \left(\operatorname{\mathcal{D}}^{(\alpha,\gamma)}_t\right)^2 T(t)  + B 
\left(\operatorname{\mathcal{D}}^{(\alpha,\gamma)}_t\right) T(t) + \lambda T(t) = 0 
\end{equation}
and 
\begin{equation*}
\label{eq.Y}
\mathcal{G}[Y] = -\lambda Y . 
\end{equation*}
We have $Y(y) = Q_n(y)$, where $Q_n(y)$ denotes the classical orthogonal polynomials 
discussed in the previous section, and eigenvalue $\lambda = \lambda_n$ is given by 
\eqref{eq.eigenvalues}. 
As mentioned before we assume that the polynomials are properly normalized in order to have
\begin{equation*}
	\int Q_n(x) Q_m(x)\boldsymbol{m}(x)\, dx = \delta_{mn}. 
\end{equation*}
Notice that the equation \eqref{eq.T} takes the form \eqref{eq.2} with parameters $a = B/A$ and $b =\lambda_n/A$. 
Therefore, we have temporal solution 
\begin{equation}\label{temporal_solution}
T_n(t; \alpha, \gamma)= K_1(n) \, \operatorname{E}_{\alpha,1+\gamma/\alpha,\gamma/\alpha}\left(a^{*}_n t^{\alpha+\gamma}\right) + 
K_2(n) \, \operatorname{E}_{\alpha,1+\gamma/\alpha,\gamma/\alpha}\left(b^*_n t^{\alpha+\gamma}\right),
\end{equation}
where 
\begin{equation*} \label{a*b*_n}
	a^{*}_n=\frac{-B}{2A}+\sqrt{\left(\frac{-B}{2A}\right)^2-\frac{\lambda_n}{A}}, \quad  b^{*}_n=\frac{-B}{2A}-\sqrt{\left(\frac{-B}{2A}\right)^2-\frac{\lambda_n}{A}}
\end{equation*}
and
\begin{equation*}\label{K1K2_n}
	 K_1(n)=\frac{a^*_n -A^{-1}\lambda_n }{a^*_n(a^*_n-b^*_n)}, \quad K_2(n)=\frac{A^{-1}\lambda_n -b^*_n}{b^*_n (a^*_n-b^*_n)}
\end{equation*}
(compare with \eqref{K1K2} and  \eqref{a*b*}, where dependence on $n$ is through eigenvalues $\lambda_n$). Recall that $K_1(n)+K_2(n)=1$. Sequence $\left(\lambda_n, \, n \in \mathbb{N}_0\right)$ is increasing sequence of eigenvalues with $\lambda_n \to \infty$ as $n \to \infty$ for each considered Pearson diffusion. Therefore constants $a^*_n$ and $b^*_n$ become complex conjugate numbers after certain threshold ($B^2/(4A^2)-\lambda_n/A <0 $). 
Finally, heuristic arguments analogous to previous section lead to the  solution of \eqref{frac.hyperbolic.pearson}
\begin{equation}
	\label{frac.hyperbolic.pearson.solution}
g(y,t)=  \sum_{n=0}^\infty a_n T_n(t; \alpha, \gamma)
	Q_n(y),
\end{equation}
where
\begin{equation*}
a_n T_n(0; \alpha, \gamma)=a_n=\int Q_n(x)g(x)\boldsymbol{m}(x)dx, \quad g(y)=g(y,0).
\end{equation*}
Next step is to formally prove that the strong solution of Cauchy problem corresponding to \eqref{frac.hyperbolic.pearson} is given by \eqref{frac.hyperbolic.pearson.solution}.
Techniques used are similar to those in \cite{Pearson}. However, due to appearance of complex numbers through constants $K_1(n)$ and $K_2(n)$, as well as through KS function (via $a^*$ and $b^*$) we will need certain adaptions from techniques used for Mittag-Leffler function in \cite{Pearson}. On the other hand, the space part of the equation (due to generator of Pearson diffusion) remains unchanged.

Taking this into account, first step is the following theorem.

\begin{theorem}
For the three classes of Pearson diffusions (OU, CIR and Jacobi) whose stationary density $\boldsymbol{m}$ and orthogonal polynomials $\left(Q_n, \, n \in \mathbb{N}_0\right)$ are given in Section \ref{PearsonDiffusions}, for any $0<\alpha+\gamma\leq1$, the series
\begin{equation*} \label{ppear}
p^{H}_{\alpha, \gamma}(x,t;y)=\boldsymbol{m}(x)\sum_{n=0}^{\infty} T_n(t;\alpha, \gamma)Q_n(x)Q_n(y)
\end{equation*}
with $T_n(\cdot; \alpha, \gamma)$ defined via \eqref{temporal_solution} converges for fixed $t>0$ and $x$, $y$ in the state space of the corresponding diffusion (OU, CIR or Jacobi).
\end{theorem}

\begin{proof} 
Taking into account Theorem \ref{KS_asymtptocis_complex_z} for parameters $a=\alpha \in \langle 0, 1\rangle, m=1+\gamma/\alpha, l=\gamma/\alpha$ and $z=-a^*_n t^{\alpha+\gamma}$ we obtain 
\begin{equation*}
	\operatorname{E}_{\alpha,1+\gamma/\alpha,\gamma/\alpha}\left(a^*_n t^{\alpha+\gamma}\right) = -\frac{1}{\Gamma(1-\alpha) a^{*}_n t^{\alpha+\gamma}}+\mathcal{O}\left(|a^*_n t^{\alpha+\gamma}|^{-\delta}\right), \quad \text{ as } n \to \infty,
\end{equation*}
where $1<\delta\leq 3/2$. In particular, we have 
\begin{equation} \label{KS_asympt_an*}
\operatorname{E}_{\alpha,1+\gamma/\alpha,\gamma/\alpha}\left(a^*_n t^{\alpha+\gamma}\right) \sim -\frac{1}{\Gamma(1-\alpha) a^{*}_n t^{\alpha+\gamma}}, \quad \text{ as } n \to \infty.
\end{equation}
Conditions of the abovementioned theorem are fulfilled since
\begin{equation*}
 \quad 0 < \alpha <1, \alpha+\gamma \leq 1
\end{equation*}
and
\begin{equation*} \label{tau}
Re(a^*_n t^{\alpha+\gamma})=-\frac{B}{2A}t^{\alpha+\gamma}<0, \quad Im(a^*_n t^{\alpha+\gamma})=t^{\alpha+\gamma}\sqrt{\frac{\lambda_n}{A}-\left(-\frac{B}{2A}\right)^2} \to \infty, \quad \text{ as } n \to \infty.
\end{equation*}
Similarly,
\begin{equation} \label{KS_asympt_bn*}
	\operatorname{E}_{\alpha,1+\gamma/\alpha,\gamma/\alpha}\left(b^*_n t^{\alpha+\gamma}\right) \sim -\frac{1}{\Gamma(1-\alpha) b^{*}_n t^{\alpha+\gamma}}, \quad \text{ as } n \to \infty.
\end{equation}
since for $n$ large enough, $a^*_n$ and $b^*_n$ are complex conjugate numbers.
It is not hard to see that
\begin{equation} \label{K1K2_asymp}
K_1(n) \sim K_2(n) \sim \frac{1}{2}, \quad n \to \infty.
\end{equation}
Taking this into account together with \eqref{KS_asympt_an*} and \eqref{KS_asympt_bn*} we have
\begin{align}
T_n(t; \alpha, \gamma) &\sim \frac{1}{2}\left(-\frac{1}{\Gamma(1-\alpha) a^{*}_n t^{\alpha+\gamma}}-\frac{1}{\Gamma(1-\alpha) b^{*}_n t^{\alpha+\gamma}}\right) \nonumber \\ \nonumber
                       &\sim -\frac{1}{2\Gamma(1-\alpha)t^{\alpha+\gamma}}\left(\frac{1}{ a^{*}_n}+\frac{1}{ b^{*}_n }\right)\\ 
                       & \sim \frac{B}{2\Gamma(1-\alpha) \lambda_n t^{\alpha+\gamma}}, \quad n \to \infty. \label{Hyper_temporal_asym}
\end{align}

Rewriting 
\begin{equation} \label{series_rep_solution_n0}
	p^{H}_{\alpha, \gamma}(x,t;y)=\boldsymbol{m}(x)\left(\sum_{n=0}^{n_0 -1} T_n(t;\alpha, \gamma)Q_n(x)Q_n(y)+\sum_{n=n_0}^{\infty} T_n(t;\alpha, \gamma)Q_n(x)Q_n(y)\right),
\end{equation}
we simply need to prove convergence of the tail series (after $n_0$). 
Since the space part of the equation \eqref{frac.hyperbolic.pearson} involves generator of OU, CIR and Jacobi diffusion in the same manner as in \cite{Pearson} we can simply use bounds for corresponding eigenfunctions (orthonormal polynomials - Hermitte, Laguerre and Jacobi) obtained there:
\begin{itemize}
	\item OU case: $\lambda_n=\theta n$, $Q_n(x) \leq K \exp(x^2/4) n^{-1/4}\left(1+|x/\sqrt{2}|^{5/2}\right)$,
	\item CIR case: $\lambda_n=\theta n$, $Q_n(x)=\mathcal{O}\left(\frac{\exp(x/2)}{x^{(2b-1)/4}}n^{-1/4}\right)$,
	\item Jacobi case: $\lambda_n=\theta n  (n+a+b+1)/(a+b+2)$, $Q_n(x)=C(x,a,b)\cos(N \theta + c)+\mathcal{O}(n^{-1})$, where $x=\cos \theta, N=n+1/2(a+b+1), c=-(a+1/2)\pi/2$
\end{itemize}
(for details, see [\cite{Pearson}, p. 536-537]).

Combining these bounds with \eqref{Hyper_temporal_asym} we have for OU and CIR case

\begin{equation} \label{tail_series_OUCIR}
|T_n(t;\alpha, \gamma)Q_n(x)Q_n(y)|\leq \frac{C(x,y,\alpha, \gamma, t, n_0)}{n^{3/2}}, \quad n \geq n_0
\end{equation}
and for Jacobi case
\begin{equation} \label{tail_series_Jacobi}
	|T_n(t;\alpha, \gamma)Q_n(x)Q_n(y)|\leq \frac{C(x,y,\alpha, \gamma, t, n_0)\cos (N \theta + c)}{n^{2}}, \quad n \geq n_0,
\end{equation}
where the constants $C(x, y, \alpha, \gamma, t, n_0)$ are not necessarily the same for each diffusion and might also depend on corresponding infinitesimal parameters $\mu(x)$ and $\sigma(x)$ of the diffusion generator $\mathcal{G}$.
Finally, inequalities \eqref{tail_series_OUCIR} and \eqref{tail_series_Jacobi} show that the series \eqref{series_rep_solution_n0} is convergent for corresponding OU, CIR and Jacobi case, which completes the proof.
\end{proof}

In the next two theorems we provide strong solutions for the corresponding fractional Cauchy problems. The proofs are based on ideas in \cite{Pearson}, while the form of solutions can be foreseen from the heuristic argument in the last section.
\begin{theorem} \label{Cauchy_problem_Gen}
Let $h \in L_2(\boldsymbol{m}(x)dx)$ be such a function that $\sum_{n}  a_n Q_n$, where $$a_n=\int_{S} h(x) Q_n(x)\boldsymbol{m}(x) dx,$$ converges to function $h$ uniformly on finite intervals (subsets of state space $S$ of corresponding diffusion). Then the fractional Cauchy problem
\begin{equation} \label{Cauchy_problem}
 A\cdot \left(\mathcal{D}^{(\alpha,\gamma)}_t\right)^2 g(t,y)+B \cdot \mathcal{D}^{(\alpha,\gamma)}_t g(t,y)= \mathcal{G} g(t,y), \quad g(0,y)=h(y) 
\end{equation}
has a strong solution given by
\begin{equation} \label{Cauchy_problem_solution}
g^{H}_{\alpha, \gamma}(t,y)=\sum_{n=0}^{\infty} T_n(t; \alpha, \gamma)Q_n(y) a_n.
\end{equation}
The series in \eqref{Cauchy_problem_solution} converges absolutely for each  fixed $t>0$, $y \in S$, and \eqref{Cauchy_problem} holds pointwise.
\end{theorem}
\begin{proof}
First, recall that $Q_n(y)$ is eigenfunction of the infinitesimal generator $\mathcal{G}$ with corresponding eigenvalue $\lambda_n$. Second, recall $T_n(t, \alpha, \gamma)$ is a temporal solution of the fractional equation \eqref{eq.2} with $a=B/A$ and $b=\lambda_n/A$. Therefore
\begin{align*}
\mathcal{G}\left[T_n(t; \alpha, \gamma) Q_n(y) a_n\right] &= T_n(t; \alpha, \gamma)a_n \left[\mathcal{G}Q_n(y)\right] \\
&=T_n(t; \alpha, \gamma)a_n \left[-\lambda_n Q_n(y)\right] \\
&= Q_n(y) a_n \left[-\lambda_n  T_n(t; \alpha, \gamma)\right]\\
&=\left[ A\cdot \left(\mathcal{D}^{(\alpha,\gamma)}_t\right)^2+B \cdot \mathcal{D}^{(\alpha,\gamma)}_t\right] T_n(t; \alpha, \gamma) Q_n(y) a_n.
\end{align*}
Therefore, each term of the series \eqref{Cauchy_problem_solution} satisfies \eqref{Cauchy_problem}.
At last, we need to check that we can differentiate series \eqref{Cauchy_problem_solution} term by term. Taking into account classical results on this matter (e.g \cite{Rudin}, Theorems 7.16. and 7.17.) we need to check absolute and uniform convergence on finite intervals of the following series
\begin{equation*}
	 \left\{
\begin{aligned}
&\sum\limits_{n=0}^{\infty}\mathcal{D}^{(\alpha,\gamma)}_t T_n(t; \alpha, \gamma) Q_n(y) a_n \\
&\sum\limits_{n=0}^{\infty}\left(\mathcal{D}^{(\alpha,\gamma)}_t\right)^2 T_n(t; \alpha, \gamma) Q_n(y) a_n \\
&\sum\limits_{n=0}^{\infty} T_n(t; \alpha, \gamma) Q'_n(y) a_n \\
&\sum\limits_{n=0}^{\infty} T_n(t; \alpha, \gamma) Q^{''}_n(y) a_n \\
&\sum\limits_{n=0}^{\infty}  Q_n(y) a_n. \\
\end{aligned}
\right.
\end{equation*}
Last series is convergent due to the assumption. Regarding the series involving operator \eqref{Caputo_variation} $\mathcal{D}^{(\alpha,\gamma)}_t$ first notice that
\begin{align*}
\mathcal{D}^{(\alpha,\gamma)}_t T_n(t; \alpha, \gamma)&=\mathcal{D}^{(\alpha,\gamma)}_t \left( K_1(n) \, \operatorname{E}_{\alpha,1+\gamma/\alpha,\gamma/\alpha}\left(a^{*}_n t^{\alpha+\gamma}\right) + 
K_2(n) \, \operatorname{E}_{\alpha,1+\gamma/\alpha,\gamma/\alpha}\left(b^*_n t^{\alpha+\gamma}\right)\right) \\
&= K_1(n) \, (-a^*_n) \operatorname{E}_{\alpha,1+\gamma/\alpha,\gamma/\alpha}\left(a^{*}_n t^{\alpha+\gamma}\right) + 
K_2(n) \, (-b^*_n) \operatorname{E}_{\alpha,1+\gamma/\alpha,\gamma/\alpha}\left(b^*_n t^{\alpha+\gamma}\right),
\end{align*} 
and 
\begin{equation*}
	\left(\mathcal{D}^{(\alpha,\gamma)}_t\right)^2 T_n(t; \alpha, \gamma)=K_1(n) \, (a^*_n)^2 \operatorname{E}_{\alpha,1+\gamma/\alpha,\gamma/\alpha}\left(a^{*}_n t^{\alpha+\gamma}\right) + 
	K_2(n) \, (b^*_n)^2 \operatorname{E}_{\alpha,1+\gamma/\alpha,\gamma/\alpha}\left(b^*_n t^{\alpha+\gamma}\right),
\end{equation*} 
where we have used the fact $\operatorname{E}_{\alpha,1+\gamma/\alpha,\gamma/\alpha}\left(-\kappa t^{\alpha+\gamma}\right)$ is the solution of eigenvalue problem \eqref{fde.1}.

Recalling asymptotic formulas \eqref{KS_asympt_an*} and \eqref{K1K2_asymp}  it follows
\begin{equation*}
\mathcal{D}^{(\alpha,\gamma)}_t T_n(t; \alpha, \gamma) \sim \frac{1}{\Gamma(1-\alpha)t^{\alpha+\gamma}}, \quad n \to \infty
\end{equation*}
and
\begin{equation*}
\left(\mathcal{D}^{(\alpha,\gamma)}_t\right)^2 T_n(t; \alpha, \gamma) \sim -\frac{B}{2A}\cdot\frac{1}{\Gamma(1-\alpha)t^{\alpha+\gamma}}, \quad n \to \infty.
\end{equation*}
Therefore, both series $\sum\limits_{n=0}^{\infty}\left(\mathcal{D}^{(\alpha,\gamma)}_t\right)^2  T_n(t; \alpha, \gamma) Q_n(y) a_n$ and $\sum\limits_{n=0}^{\infty}\mathcal{D}^{(\alpha,\gamma)}_t T_n(t; \alpha, \gamma) Q_n(y) a_n$ are convergent due to convergence assumption of the series $\sum\limits_{n=0}^{\infty} Q_n(y) a_n$.
For the series involving first and second order derivatives of eigenfunctions the proof is analogous as in the proof of [Theorem 3.2, \cite{Pearson}] since temporal solution $T_n(t; \alpha, \gamma)$ is of order $\mathcal{O}(\lambda^{-1}_n t^{-\alpha+\gamma})$. 
Therefore, we conclude that we can differentiate series \eqref{Cauchy_problem_solution} term by term for all three diffusions and it satisfies corresponding fractional equation given in \eqref{Cauchy_problem}.
Lastly, the initial condition is fulfilled since 
\begin{align*}
	g^{H}_{\alpha, \gamma}(0, y)&=\sum_{n=0}^{\infty}T_n(0; \alpha, \gamma)Q_n(y)a_n \\
	&=\sum_{n=0}^{\infty}\left(K_1(n)+K_2(n)\right)Q_n(y)a_n \\
	&=\sum_{n=0}^{\infty}Q_n(y)a_n \\
	&=h(y).
\end{align*}
\end{proof}

\begin{theorem}\label{Cauchy_problem_FP}
	Let $h/\boldsymbol{m}(x) \in L_2(\boldsymbol{m}(x)dx)$ be such a function that $\sum_{n} a_n Q_n$, where $$a_n=\int_{S} h(y) Q_n(y) dy,$$ converges to function $h/\boldsymbol{m}$ uniformly on finite intervals (subsets of state space $S$ of corresponding diffusion). Then the fractional Cauchy problem
	\begin{equation} \label{Cauchy_problem2}
		A\cdot \left(\mathcal{D}^{(\alpha,\gamma)}_t\right)^2 g(t,x)+B \cdot \mathcal{D}^{(\alpha,\gamma)}_t g(t,x)= \mathcal{L} g(t,x), \quad g(0,x)=h(x) 
	\end{equation}
	has a strong solution given by
	\begin{equation} \label{Cauchy_problem_solution2}
		g^{H}_{\alpha, \gamma}(t,x)=\boldsymbol{m}(x)\sum_{n=0}^{\infty} T_n(t; \alpha, \gamma)Q_n(x) a_n.
	\end{equation}
	The series in \eqref{Cauchy_problem_solution2} converges absolutely for each  fixed $t>0$, $x \in S$, and \eqref{Cauchy_problem2} holds pointwise.
\end{theorem}
\begin{proof}
First notice
\begin{align*}
\mathcal{L}\left(\boldsymbol{m}(x) T_n(t; \alpha, \gamma)Q_n(x) a_n\right)&=T_n(t; \alpha, \gamma) a_n\left[-\frac{d}{dx}\left(\mu(x)\boldsymbol{m}(x)Q_n(x)\right)+\frac{d}{d x^2}\left(\sigma^2(x)\boldsymbol{m}(x)Q_n(x)/2\right)\right] \\
&=T_n(t; \alpha, \gamma) \cdot a_n \cdot \boldsymbol{m}(x) \cdot\mathcal{G}Q_n(x) \\
&=-\lambda_n T_n(t; \alpha, \gamma) \cdot a_n \cdot \boldsymbol{m}(x)\cdot Q_n(x) \\
&=\left[A\cdot \left(\mathcal{D}^{(\alpha,\gamma)}_t\right)^2 g(t,y)+B \cdot \mathcal{D}^{(\alpha,\gamma)}_t \right]\boldsymbol{m}(x) T_n(t; \alpha, \gamma)Q_n(x) a_n.
\end{align*}
We conclude that each term of the series \eqref{Cauchy_problem_solution2} satisfies equation \eqref{Cauchy_problem2}. Next, we need to prove that we can differentiate the series term by term, i.e. we need to check absolute and uniform convergence of the following series:
\begin{equation*}
	\left\{
	\begin{aligned}
		&\sum\limits_{n=0}^{\infty}\mathcal{D}^{(\alpha,\gamma)}_t T_n(t; \alpha, \gamma) Q_n(x) a_n \\
		&\sum\limits_{n=0}^{\infty}\left(\mathcal{D}^{(\alpha,\gamma)}_t\right)^2 T_n(t; \alpha, \gamma) Q_n(x) a_n \\
		&\sum\limits_{n=0}^{\infty} T_n(t; \alpha, \gamma) Q'_n(x) a_n \\
		&\sum\limits_{n=0}^{\infty} T_n(t; \alpha, \gamma) Q^{''}_n(x) a_n \\
		&\sum\limits_{n=0}^{\infty} Q_n(x) a_n. \\
	\end{aligned}
	\right.
\end{equation*}
The proof follows similarly to that of the previous theorem, and as such, we omit its details. Finally, the initial condition is fulfilled since
\begin{align*}
	g^{H}_{\alpha, \gamma}(0, y)&=\boldsymbol{m}(x)\sum_{n=0}^{\infty}T_n(0; \alpha, \gamma)Q_n(x)a_n \\
	&=\boldsymbol{m}(x)\sum_{n=0}^{\infty}\left(K_1(n)+K_2(n)\right)Q_n(x)a_n \\
	&=\boldsymbol{m}(x)\sum_{n=0}^{\infty}Q_n(x)a_n \\
	&=\boldsymbol{m}(x) \left(h(x)/\boldsymbol{m}(x)\right) \\
	&=h(x).
\end{align*}
\end{proof}

\begin{remark}
It should be noted that the solutions of the last two theorems are unique. Indeed, let $g_1$ and $g_2$ be different solutions of the corresponding fractional Cauchy problem. Then by the initial condition $g_1(0,z)=g_2(0,z)=h(z)$. Since $g=g_1-g_2$ is also a solution with initial point condition $g(0,z)=g_1(0,z)-g_2(0,z)=h(z)-h(z)=0$ this implies every $a_n=0$ and so $g\equiv0$, providing desired conclusion.
\end{remark}
\begin{remark}
Last three theorems can bee seen as extension of the Theorems 3.1. - 3.3. in \cite{Pearson}. In order to see that set $\gamma=0, A=0, B=1$ in \eqref{frac.hyperbolic.pearson}.
\end{remark}

\section{Stochastic representation}\label{Stoch_rep_sec}
In this section we provide stochastic representations of analytical solutions of Cauchy problems from last two sections.

From [\cite{Boudabsa}, equation (3.1)], the following representation of the Kilbas-Saigo function holds, for $\alpha \in (0,1)$, $m>0$ and for any $z \in \mathbb{C}$:

\begin{equation*}
	\operatorname{E}_{\alpha, m, m-1}(z)=\mathbb{E}\exp\left(z \int_{0}^\infty \left(1-\sigma^{\alpha}_t\right)^{\alpha(m-1)}_{+}dt\right),
\end{equation*}
where $\left(\sigma^{\alpha}_t, t \geq 0\right)$ is $\alpha$-stable subordinator such that
\begin{equation*}
	\mathbb{E}\exp(-\lambda \sigma^{\alpha}_t)=\exp(-t \lambda^\alpha), \quad t \geq 0.
\end{equation*}
Therefore, we can see Kilbas-Saigo function as Laplace transform of some nonnegative variable, however not in the form of inverse subordinator as we might expect.
	Thus we can write, for $\alpha \in (0,1)$ and $\gamma > -\alpha$, that
\begin{equation}\label{epr}
	\operatorname{E}_{\alpha,1+\gamma/\alpha,\gamma/\alpha}\left(-\lambda t^{\alpha+\gamma}\right)=\mathbb{E}\exp\left(-\lambda t^{\alpha+\gamma}\int_{0}^\infty (1-\sigma_s^\alpha)_+^{\gamma}ds\right)=\mathbb{E}\exp\left(-\lambda Z^{(\alpha,\gamma)}_t\right),
\end{equation}
where the random process $(Z^{(\alpha,\gamma)}_t, t \geq 0)$ is defined as $Z^{(\alpha,\gamma)}_t:=t^{\alpha+\gamma}Z$, $t \geq0,$ for the r.v. $Z:=\int_{0}^\infty (1-\sigma_t^\alpha)_+^{\gamma}ds$. Moreover, the following equality in distribution is proved in \cite{Simon} for $Z$, in terms of infinite independent product of Beta r.v.'s 
$\mathcal{B}(a,b)$: 

\begin{equation*}
	Z \overset{d}{=}\frac{\Gamma(\gamma+1)}{\Gamma(\alpha+\gamma+1)}\prod_{n=0}^\infty \frac{\gamma +n+1}{\alpha+\gamma+n}\mathcal{B}\left(1+\frac{n}{\alpha+\gamma},\frac{1-\alpha}{\alpha+\gamma}\right),\label{beta}
\end{equation*}
which, for $\gamma=0$ and $\alpha \in (0,1)$, reduces to the following equality in distribution: $Z \overset{d}{=}\mathcal{L}_1^{\alpha}$, where $(\mathcal{L}_t^{\alpha},t \geq 0)$ is the inverse of $\sigma_t^\alpha$.	
\begin{definition}\label{defpear}
	Let us denote by $\left(X_t, \,t \geq 0\right)$ the standard Pearson diffusion with transition density $\boldsymbol{p}(x,t;y)$  
	given in \eqref{denpear}, for $x,y \in \mathbb{R}$ and $t \geq 0$. Then we define the stretched non-local Pearson diffusion $\left(X^{({\alpha,\gamma})}_t, \,t \geq 0\right)$ as 
	\begin{equation} 
		X^{({\alpha,\gamma})}_t:=X_{Z^{(\alpha,\gamma)}_t}, \qquad t \geq 0, \label{ed}
	\end{equation}
	under the assumption that $X_t$ and $Z$ are independent for every $t$.
\end{definition}
 The process $\left(X^{({\alpha,\gamma})}_t, \, t \geq 0\right)$ is well-defined according to Def. \ref{defpear}, since $\left(Z^{(\alpha,\gamma)}_t ,\, t \geq 0\right)$ is, by definition, non-decreasing almost surely and such that $Z^{(\alpha,\gamma)}_0\overset{a.s.}{=}0.$ Since the process defined in \eqref{ed} is non-Markovian, we say that it has transition density $\boldsymbol{p}_{\alpha,\gamma}(x,t;y)$ if
\[
P\left(\left.X^{({\alpha,\gamma})}_t \in B \right\vert X^{({\alpha,\gamma})}_0=y\right)=\int_B \boldsymbol{p}_{\alpha,\gamma}(x,t;y)dx, 
\]
for any Borel set $B$. Note that we have denoted by $X_t$ the process $X^{(1,0)}_t$, for brevity.

Now, interesting particular cases of our model
\begin{equation*}
	A \left(\operatorname{\mathcal{D}}^{(\alpha,\gamma)}_t\right)^2 g(y,t)  + B 
	\left(\operatorname{\mathcal{D}}^{(\alpha,\gamma)}_t\right) g(y,t)  = 
	\mathcal{G}[g] 
\end{equation*}
are
\begin{itemize}
	\item $A=0, B=1, \gamma=0$ leading to fractional (Pearson) diffusion model (see \cite{Pearson})
	\item $A=0, B=1, \gamma>0$ leading to new time changed diffusion model via Kilbas-Saigo function (see Section \ref{sec_fpd} for analytical solution)
	\item $A=1, B=2h>0$  leading to results based on spectral theory of infinitesimal generator of the corresponding diffusion, where time-part of the equation resembles (fractional) telegraph equation, while space-part of the equation is governed by Pearson diffusion (see Section \ref{sec_fhpd} for analytical solution). This should be compared to the d'Alembert formula for fractional telegraph equation (see \cite{frac_dAlamb2} and \cite{frac_dAlamb}).
\end{itemize}
The next step is to provide stochastic representation of the obtained analytical solutions of the proposed model.
\subsection{Stochastic representation for stretched non-local Pearson diffusion}

\begin{theorem}
	The transition density of the stretched non-local Pearson diffusion $(X^{(\alpha,\gamma)}_t,t \geq 0)$ coincides with  $\boldsymbol{p}_{\alpha, \gamma}(x,t;y)$ given in \eqref{trans_fpd}.
\end{theorem}
\begin{proof}
	Let us denote by $f_t(\cdot)$ the density of $Z^{(\alpha,\gamma)}_t$, for any $t \geq 0$, then, by considering \eqref{ed} and by the Fubini theorem, we can write that 
	\begin{eqnarray}\label{Fubini_FPD}
		P\left(\left.X^{({\alpha,\gamma})}_t \in B \right\vert X^{({\alpha,\gamma})}_0=y\right)&=&\int_0^\infty P\left(\left.X_\tau \in B \right\vert X_0=y\right) f_t(\tau)d\tau \label{inter} \\
		&=& \int_0^\infty \int_B\boldsymbol{p}(x,\tau;y)f_t(\tau)dxd\tau \notag \\
		&=&\int_B \boldsymbol{m}(x)\sum_{n=0}^{\infty} Q_n(x)Q_n(y) \int_0^\infty \exp(-\lambda_n \tau)f_t(\tau)d\tau dx, \notag \\
		&=&\int_B \boldsymbol{m}(x)\sum_{n=0}^{\infty} \operatorname{E}_{\alpha,1+\gamma/\alpha,\gamma/\alpha}\left(-\lambda_n t^{\alpha+\gamma}\right) Q_n(x)Q_n(y) dx, \notag \\
		&=&  \int_B\boldsymbol{p}_{\alpha,\gamma}(x,t;y) dx, \notag
	\end{eqnarray}
	where in the third equality we have used the spectral representation of transition density of Pearson diffusion  \eqref{denpear}, while in  fourth equality we have used Laplace transform \eqref{epr}.
The interchange of integral and sum in \eqref{Fubini_FPD} can be justified by Fubini argument by taking into account Theorem \ref{trans_dens_steched_Pearson}.
\end{proof}
Next two theorems provide stochastic representation results for Theorem \ref{Cauchy_problem_Gen_streched} and Theorem \ref{Cauchy_problem_FP_streched}, i.e.~for fractional Cauchy problems \eqref{Cauchy_problem_fpd} and \eqref{Cauchy_problem2_fpd}.

\begin{theorem}
For any function $h$ that satisfies conditions of Theorem \ref{Cauchy_problem_Gen_streched}, the function
\begin{equation*}
g_{\alpha, \gamma}^{H}(t,y)=\mathbb{E} \left[h\left(X_t^{(\alpha, \gamma)}\right)\, \Big| \, X_0^{(\alpha, \gamma)}=y\right]
\end{equation*}
solves the fractional Cauchy problem \eqref{Cauchy_problem_fpd}.
\end{theorem}
\begin{proof}
Taking into account last theorem together with \eqref{trans_fpd} we can rewrite the expectation as:
\begin{align}
	g_{\alpha, \gamma}^{H}(t,y)&=\mathbb{E} \left[h\left(X_t^{(\alpha, \gamma)}\right)\, \Big| \, X_0^{(\alpha, \gamma)}=y\right] \nonumber\\
	                           &=\int h(x) \boldsymbol{p}_{\alpha,\gamma}(x,t;y) dx \nonumber\\
	                           &=\int h(x) \boldsymbol{m}(x)\sum_{n=0}^{\infty} \operatorname{E}_{\alpha,1+\gamma/\alpha,\gamma/\alpha}\left(-\lambda_n t^{\alpha+\gamma}\right) Q_n(x)Q_n(y) dx \nonumber\\
	                           &=\sum_{n=0}^{\infty}  \operatorname{E}_{\alpha,1+\gamma/\alpha,\gamma/\alpha}\left(-\lambda_n t^{\alpha+\gamma}\right) Q_n(y)  \left( \int Q_n(x) h(x) \boldsymbol{m}(x) dx \right). \label{Stoch_sol_Gen_form}
\end{align}
In the final equality, the interchange of the order of integration and summation is justified by an application of Fubini’s theorem. Specifically, we consider
\begin{align*}
	\int \sum_{n=0}^{\infty} |h(x)| \boldsymbol{m}(x)& \operatorname{E}_{\alpha,1+\gamma/\alpha,\gamma/\alpha}\left(-\lambda_n t^{\alpha+\gamma}\right) |Q_n(x)| |Q_n(y)| dx \\ 
	  &= 	 \sum_{n=0}^{\infty}\int |h(x)| \boldsymbol{m}(x)  \operatorname{E}_{\alpha,1+\gamma/\alpha,\gamma/\alpha}\left(-\lambda_n t^{\alpha+\gamma}\right) |Q_n(x)| |Q_n(y)| dx \\
	&=\sum_{n=0}^{\infty}  \operatorname{E}_{\alpha,1+\gamma/\alpha,\gamma/\alpha}\left(-\lambda_n t^{\alpha+\gamma}\right) |Q_n(y)|  \left( \int |Q_n(x)| |h(x)| \boldsymbol{m}(x) dx \right) \\
	& <\infty.
\end{align*}
To justify the final step and ensure convergence of the series, observe that by the Cauchy–Schwarz inequality and the fact that the orthogonal polynomials $Q_n$ are normalized in $L^2(\boldsymbol{m}(x)dx)$, we have

\begin{align*}
	\int |Q_n(x)| |h(x)| \boldsymbol{m}(x) dx & \leq \sqrt{\int |Q_n(x)|^2 \boldsymbol{m}(x) dx} \cdot \sqrt{\int |h(x)|^2 \boldsymbol{m}(x) dx} \\
	 & \leq\sqrt{\int |h(x)|^2 \boldsymbol{m}(x) dx} \\
	 & < \infty,	
\end{align*}
since $h \in L^2(\boldsymbol{m}(x)dx)$ by assumption.
Finally, the asymptotic behavior of Hermite, Laguerre and Jacobi polynomials [see \cite{Pearson}, pp. 536–537], together with the bounds for Kilbas-Saigo function provided in \eqref{KS_function_bounds}, ensures the absolute convergence of the series, thereby completing the Fubini argument.

The obtained form \eqref{Stoch_sol_Gen_form} is exactly the solution \eqref{Cauchy_problem_solution_fpd} of fractional Cauchy problem \eqref{Cauchy_problem_fpd} provided in Theorem \ref{Cauchy_problem_Gen_streched}.
\end{proof}

\begin{theorem}
	For any function $h$ that satisfies conditions of Theorem \ref{Cauchy_problem_FP_streched}, the function
	\begin{equation} \label{Stoch_sol_FP}
		p_{\alpha, \gamma}(t,x)=\int \boldsymbol{p}_{\alpha,\gamma}(x,t;y) f(y)dy
	\end{equation}
	solves the fractional Cauchy problem \eqref{Cauchy_problem2_fpd}.
\end{theorem}
\begin{proof}
Taking into account series representation of  the transition density \eqref{trans_fpd} with Fubini argument yields
\begin{align*}
	p_{\alpha, \gamma}(t,x)&=\int \boldsymbol{p}_{\alpha,\gamma}(x,t;y) f(y)dy\\
	&=\int f(y) \boldsymbol{m}(x)\sum_{n=0}^{\infty} \operatorname{E}_{\alpha,1+\gamma/\alpha,\gamma/\alpha}\left(-\lambda_n t^{\alpha+\gamma}\right) Q_n(x) Q_n(y) dy \\
	&=\sum_{n=0}^{\infty}  \operatorname{E}_{\alpha,1+\gamma/\alpha,\gamma/\alpha}\left(-\lambda_n t^{\alpha+\gamma}\right) Q_n(x)  \boldsymbol{m}(x)\left( \int Q_n(y) f(y)  dy \right),
\end{align*}
which is exactly the solution \eqref{Cauchy_problem_solution2_fpd} of fractional Cauchy problem \eqref{Cauchy_problem2_fpd} provided in Theorem \ref{Cauchy_problem_FP_streched}. The last equality can be justified in the same manner as in the last theorem, so we skip this part of the proof.
\end{proof}

In order to provide clear stochastic representation of the last theorem we need specific choice for initial function. In particular, if $f$ is initial density function of the process, i.e. probability density function of $X^{(\alpha, \gamma)}_{0}$, then the solution \eqref{Stoch_sol_FP} of the fractional Fokker-Planck Cauchy problem  is in fact
probability density function of $X^{(\alpha, \gamma)}_{t}$. For instance, we can choose $f=\boldsymbol{m}$, but even then the process will not be stationary due to random time change (the process is non-Markovian). 
However, stretched non-local Pearson diffusion have the same limiting distribution as its no-time-change counterpart.

\begin{theorem}
Let $(X_t, \, t \geq 0)$ be Ornstein-Uhlenbeck process, i.e.~diffusion process with infinitesimal mean and variance $$\mu(x)=-\theta (x-\mu), \quad \sigma^2(x)=2\theta \sigma^2,$$
where $\theta>0, \mu \in \mathbb{R}, \sigma^2>0$. Then, the stretched non-local Ornstein-Uhlenbeck process $(X^{(\alpha, \gamma)}_t, \, t \geq 0)$ defined via $\eqref{ed}$ with initial density $f(x)$ satisfying conditions of Theorem \ref{Cauchy_problem_FP_streched}  has density $p_{\alpha, \gamma}(t,x)$ such that
\begin{equation*}
p_{\alpha, \gamma}(t,x) \to \frac{1}{\sqrt{2\pi \sigma^2}}\exp\left(-\frac{(x-\mu)^2}{2\sigma^2}\right), \quad \text{ as } t \to \infty.
\end{equation*}
\end{theorem}

\begin{proof}
According to Theorem \ref{Cauchy_problem_FP_streched}, solution of fractional Fokker-Planck Cauchy problem
\begin{equation*}
		 \mathcal{D}^{(\alpha,\gamma)}_t g(t,x)= \mathcal{L} g(t,x)=\frac{\partial}{\partial x}\left(\theta (x-\mu)g(t,x)\right)+\frac{\partial^2}{\partial x^2}\left(\sigma^2 \theta g(t,x)\right), \quad g(0,x)=f(x)
\end{equation*}
is given by the series
\begin{equation*}
\boldsymbol{m}(x)\sum_{n=0}^{\infty} \operatorname{E}_{\alpha,1+\gamma/\alpha,\gamma/\alpha}\left(-\lambda_n t^{\alpha+\gamma}\right)Q_n(x) a_n,
\end{equation*}
where
\begin{align*}
\boldsymbol{m}(x)&=\frac{1}{\sqrt{2\pi \sigma^2}}\exp\left(-\frac{(x-\mu)^2}{2\sigma^2}\right), \quad Q_n(x)=\frac{(-1)^n}{c_n}\frac{\sigma^n}{\sqrt{n!}}\exp\left(\frac{(x-\mu)^2}{2\sigma^2}\right)\frac{d^n}{dx^n}\left(\exp\left(\frac{-(x-\mu)^2}{2\sigma^2}\right)\right), \\
\lambda_n&=\theta n, \quad  a_n=\int f(y) Q_n(y) dy.
\end{align*}
On the other hand, since $f(x)$ is initial density of $X^{(\alpha, \gamma)}_0$,
the density $p_{\alpha, \gamma}(t,x)$ of $X^{(\alpha, \gamma)}_t$ is exactly the given series, i.e. 
\begin{align}
	p_{\alpha, \gamma}(t,x)&=\boldsymbol{m}(x)\sum_{n=0}^{\infty} \operatorname{E}_{\alpha,1+\gamma/\alpha,\gamma/\alpha}\left(-\lambda_n t^{\alpha+\gamma}\right)Q_n(x) a_n  \nonumber\\
	&=\boldsymbol{m}(x)+\sum_{n=1}^{\infty}\boldsymbol{m}(x) \operatorname{E}_{\alpha,1+\gamma/\alpha,\gamma/\alpha}\left(-\lambda_n t^{\alpha+\gamma}\right)Q_n(x) a_n. \label{density_streched_OU}
\end{align}
Second equality follows from the fact that $\lambda_0=0, \,Q_0(x)=1$ and
$$ \operatorname{E}_{\alpha,1+\gamma/\alpha,\gamma/\alpha} \left(-\lambda_0 t^{\alpha+\gamma}\right)=1, \quad a_0=\int f(y)Q_0(y)dy=\int f(y)dy=1.$$
Next, let $t \to \infty$ in \eqref{density_streched_OU} so that
\begin{align*}
	\lim_{t \to \infty}p_{\alpha, \gamma}(t,x)&=\boldsymbol{m}(x)+\lim_{t \to \infty}\sum_{n=1}^{\infty}\boldsymbol{m}(x) \operatorname{E}_{\alpha,1+\gamma/\alpha,\gamma/\alpha}\left(-\lambda_n t^{\alpha+\gamma}\right)Q_n(x) a_n \\
	&=\boldsymbol{m}(x)+\sum_{n=1}^{\infty}\boldsymbol{m}(x) \left(\lim_{t \to \infty}\operatorname{E}_{\alpha,1+\gamma/\alpha,\gamma/\alpha}\left(-\lambda_n t^{\alpha+\gamma}\right)\right)Q_n(x) a_n \\
	&=\boldsymbol{m}(x), 
\end{align*}
where second equality follows from dominated convergence argument, while last equality follows from Kilbas-Saigo function asymptotics (see \eqref{KS_realz_asympt}). To justify the use of dominated convergence theorem, observe that from \eqref{KS_function_bounds} and [\cite{Pearson}, (3.9)] we have 
\begin{align*}
\big | \operatorname{E}_{\alpha,1+\gamma/\alpha,\gamma/\alpha}\left(-\lambda_n t^{\alpha+\gamma}\right)Q_n(x) a_n  \big | &\leq \frac{\Gamma(1+\alpha+\gamma)}{\Gamma(1+\gamma)}\cdot \frac{1}{\lambda_n \cdot t^{\alpha+\gamma}}\cdot K \exp(x^2/4)n^{-1/4} \left(1+|x/\sqrt{2}|^{5/2}\right) \cdot |a_n| \\
& \leq \frac{K}{\theta} \cdot \frac{\Gamma(1+\alpha+\gamma)}{\Gamma(1+\gamma)} \cdot \exp(x^2/4)\left(1+|x/\sqrt{2}|^{5/2}\right) \cdot \frac{|a_n|}{n^{5/4}} \cdot \frac{1}{t^{\alpha+\gamma}}.
\end{align*}
Notice that
\begin{equation*}
\sum_{n=1}^{\infty} \frac{|a_n|}{n^{5/4}} \leq \sum_{n=0}^{\infty} |a_n|^2 \cdot \sum_{n=1}^{\infty} \frac{1}{n^{5/2}} < \infty
\end{equation*}
since by Parseval's theorem 
\begin{equation*}
	\sum_{n=0}^{\infty} |a_n|^2=\int |f(y)/\boldsymbol{m}(y)|^2 \boldsymbol{m}(y) dy < \infty.
\end{equation*}
Therefore, dominated convergence argument is valid and the proof is finished.
\end{proof}

\begin{theorem}
	Let $(X_t, \, t \geq 0)$ be Cox-Ingersoll-Ross process, i.e.~diffusion process with infinitesimal mean and variance $$\mu(x)=-\theta \left(x-\frac{b}{a}\right), \quad \sigma^2(x)= \frac{2 \theta x}{a},$$
	where $\theta>0, a>0$ and $b>0$. Then, the stretched non-local Cox-Ingersoll-Ross process $(X^{(\alpha, \gamma)}_t, \, t \geq 0)$ defined via $\eqref{ed}$ with initial density $f(x)$ satisfying conditions of Theorem \ref{Cauchy_problem_FP_streched}  has density $p_{\alpha, \gamma}(t,x)$ such that
	\begin{equation*}
		p_{\alpha, \gamma}(t,x) \to \frac{a^b}{\Gamma(b)}x^{b-1}\exp(-ax), \quad \text{ as } t \to \infty.
	\end{equation*}
\end{theorem}

\begin{proof}
	According to Theorem \ref{Cauchy_problem_FP_streched}, solution of fractional Fokker-Planck Cauchy problem
	\begin{equation*}
		\mathcal{D}^{(\alpha,\gamma)}_t g(t,x)= \mathcal{L} g(t,x)=\frac{\partial}{\partial x}\left(\theta \left(x-\frac{b}{a}\right)g(t,x)\right)+\frac{\partial^2}{\partial x^2}\left( \theta \frac{x}{a} g(t,x)\right), \quad g(0,x)=f(x)
	\end{equation*}
	is given by the series
	\begin{equation*}
		\boldsymbol{m}(x)\sum_{n=0}^{\infty} \operatorname{E}_{\alpha,1+\gamma/\alpha,\gamma/\alpha}\left(-\lambda_n t^{\alpha+\gamma}\right)L^{b-1}_n(x) a_n,
	\end{equation*}
	where
	\begin{equation*}
		\boldsymbol{m}(x)=\frac{a^b}{\Gamma(b)}x^{b-1}\exp(-ax), \, L^{b-1}_n(x) = \frac{x^{-b+1}\operatorname{e}^x}{c_n \cdot n!} \frac{d^n\;}{dx^n}\left( 
		x^{b+n-1}\operatorname{e}^{-x}\right), \, \lambda_n=\theta n, \, a_n=\int f(y) L^{b-1}_n(y) dy.
	\end{equation*}
	On the other hand, since $f(x)$ is initial density of $X^{(\alpha, \gamma)}_0$,
	the density $p_{\alpha, \gamma}(t,x)$ of $X^{(\alpha, \gamma)}_t$ is exactly the given series, i.e. 
	\begin{equation*}
		p_{\alpha, \gamma}(t,x)=\boldsymbol{m}(x)\sum_{n=0}^{\infty} \operatorname{E}_{\alpha,1+\gamma/\alpha,\gamma/\alpha}\left(-\lambda_n t^{\alpha+\gamma}\right)L^{b-1}_n(x) a_n.
	\end{equation*}
	Taking into account Laguerre polynomial bound (see [\cite{Sansone}, p. 348] ) the rest of the proof is analogous to the last theorem.
\end{proof}

\begin{theorem}
	Let $(X_t, \, t \geq 0)$ be Jacobi process, i.e.~diffusion process with infinitesimal mean and variance $$\mu(x)=-\theta \left(x-\frac{b-a}{a+b+2}\right), \quad \sigma^2(x)= \frac{2\theta}{a+b+2}\left(1-x^2\right),$$
	where $\theta>0$, $a>0$ and $b>0$. Then, the stretched non-local Jacobi process $(X^{(\alpha, \gamma)}_t, \, t \geq 0)$ defined via $\eqref{ed}$ with initial density $f(x)$ satisfying conditions of Theorem \ref{Cauchy_problem_FP_streched}  has density $p_{\alpha, \gamma}(t,x)$ such that
	\begin{equation*}
		p_{\alpha, \gamma}(t,x) \to \frac{(1-x)^a(1+x)^b}{B(a+1,b+1) 2^{a+b+1}}, \quad \text{ as } t \to \infty.
	\end{equation*}
\end{theorem}

\begin{proof}
	According to Theorem \ref{Cauchy_problem_FP_streched}, solution of fractional Fokker-Planck Cauchy problem
	\begin{equation*}
		\mathcal{D}^{(\alpha,\gamma)}_t g(t,x)= \mathcal{L} g(t,x)=\frac{\partial}{\partial x}\left(\theta \left(x-\frac{b-a}{a+b+2}\right)g(t,x)\right)+\frac{\partial^2}{\partial x^2}\left( \frac{\theta \left(1-x^2\right)}{a+b+2} g(t,x)\right), \quad g(0,x)=f(x)
	\end{equation*}
	is given by the series
	\begin{equation*}
		\boldsymbol{m}(x)\sum_{n=0}^{\infty} \operatorname{E}_{\alpha,1+\gamma/\alpha,\gamma/\alpha}\left(-\lambda_n t^{\alpha+\gamma}\right)P^{(a,b)}_n(x) a_n,
	\end{equation*}
	where
	\begin{align*}
		\boldsymbol{m}(x)&=\frac{a^b}{\Gamma(b)}x^{b-1}\exp(-ax), \quad P^{(a,b)}_n(x) = \frac{(-1)^n (1-x)^{-a}(1+x)^{-b}}{c_n \cdot 2^n \cdot n!} \frac{d^n\;}{dx^n}  
		\left[(1-x)^{a+n}(1+x)^{b+n}\right], \\ \lambda_n &=\theta n, \quad a_n=\int f(y) P^{(a,b)}_n(y) dy.
	\end{align*}
	On the other hand, since $f(x)$ is initial density of $X^{(\alpha, \gamma)}_0$,
	the density $p_{\alpha, \gamma}(t,x)$ of $X^{(\alpha, \gamma)}_t$ is exactly the given series, i.e. 
	\begin{equation*}
		p_{\alpha, \gamma}(t,x)=\boldsymbol{m}(x)\sum_{n=0}^{\infty} \operatorname{E}_{\alpha,1+\gamma/\alpha,\gamma/\alpha}\left(-\lambda_n t^{\alpha+\gamma}\right)P^{(a,b)}_n(x) a_n.
	\end{equation*}
	Taking into account Jacobi polynomial bound (see [\cite{Pearson}, (3.10)] ) the rest of the proof is analogous to the OU case.
\end{proof}

\subsection{Stochastic representation for fractional hyperbolic Pearson diffusion}

Next, we explore the probabilistic connection between fractional Cauchy problems
\eqref{Cauchy_problem}, \eqref{Cauchy_problem2} and their corresponding non-fractional counterparts (i.e.~these Cauchy problems with $\alpha=1, \gamma=0$).

\begin{theorem}
		For any function $h$ that satisfies conditions of Theorem \ref{Cauchy_problem_Gen}, the function
			\begin{equation*}
		g_{\alpha, \gamma}^{H}(t,y)=\mathbb{E} \left[g_{1, 0}^{H}(Z^{(\alpha, \gamma)}_t,y)\right]
	\end{equation*}
	solves the fractional Cauchy problem \eqref{Cauchy_problem}, where $g_{\alpha, \gamma}^{H}(t,y)$ is given by \eqref{Cauchy_problem_solution}.
\end{theorem}
\begin{proof}
First, let us recall the temporal solution (see \eqref{temporal_solution})
\begin{equation*}
T_n(t; \alpha, \gamma)= K_1(n) \, \operatorname{E}_{\alpha,1+\gamma/\alpha,\gamma/\alpha}\left(a^{*}_n t^{\alpha+\gamma}\right) + 
K_2(n) \, \operatorname{E}_{\alpha,1+\gamma/\alpha,\gamma/\alpha}\left(b^*_n t^{\alpha+\gamma}\right)
\end{equation*}
of 
\begin{equation*}
 A \left(\operatorname{\mathcal{D}}^{(\alpha,\gamma)}_t\right)^2 T(t)  + B 
\left(\operatorname{\mathcal{D}}^{(\alpha,\gamma)}_t\right) T(t) =- \lambda T(t). 
\end{equation*}
On the other hand, taking into account stochastic representation of Kilbas-Saigo function (see \eqref{epr}) we have
\begin{align*}
	T_n(t; \alpha, \gamma)&= K_1(n) \, \mathbb{E}\exp\left(a^*_n Z^{(\alpha,\gamma)}_t\right) + 
	K_2(n) \, \mathbb{E}\exp\left(b^*_n Z^{(\alpha,\gamma)}_t\right) \\
	                      &=\mathbb{E}\left[K_1(n)\exp\left(a^*_n Z^{(\alpha,\gamma)}_t\right)+K_2(n)\exp\left(b^*_n Z^{(\alpha,\gamma)}_t\right)\right] \\
	                      &=\mathbb{E}\left[T_n(Z^{(\alpha,\gamma)}_t, 1, 0)\right].
\end{align*}
To see this, let $\alpha=1$ and $\gamma=0$ so that $Z^{(\alpha,\gamma)}_t=t$ a.s.~and
\begin{equation*}
T_n(t,1,0)=\mathbb{E}\left[K_1(n)\exp(a^*_n t)+K_2(n)\exp(b^*_n t)\right].
\end{equation*}
which is temporal solution of
\begin{equation*}
	A \cdot  \frac{d^2}{dt^2}T(t)  + B \cdot
	 \frac{d}{dt}T(t) =- \lambda T(t). 
\end{equation*}
Now, solution of fractional Cauchy problem \eqref{Cauchy_problem}, i.e.
\begin{equation*} 
	A\cdot \left(\mathcal{D}^{(\alpha,\gamma)}_t\right)^2 g(t,y)+B \cdot \mathcal{D}^{(\alpha,\gamma)}_t g(t,y)= \mathcal{G} g(t,y), \quad g(0,y)=h(y) 
\end{equation*}
can be written as
\begin{align*}
	g_{\alpha, \gamma}^{H}(t,y)&=\sum_{n=0}^{\infty} T_n(t; \alpha, \gamma)Q_n(y) a_n. \\
	&=\sum_{n=0}^{\infty} \mathbb{E}\left[T_n(Z^{(\alpha,\gamma)}_t, 1, 0)\right]Q_n(y) a_n \\
	&=\mathbb{E}\left[\sum_{n=0}^{\infty} T_n(Z^{(\alpha,\gamma)}_t, 1, 0) Q_n(y) a_n\right] \\
	&=\mathbb{E} \left[g_{1, 0}^{H}(Z^{(\alpha, \gamma)}_t,y)\right],
\end{align*}
where the third line is guaranteed by Fubini argument similar to the previous theorems. In the last equality $g_{1, 0}^{H}(t,y)$ is the solution of the non-fractional Cauchy problem
\begin{equation*} 
	A\cdot \frac{d^2}{dt^2} g(t,y)+B \cdot \frac{d}{dt} g(t,y)= \mathcal{G} g(t,y), \quad g(0,y)=h(y). 
\end{equation*}
Therefore, the proof is finished.
\end{proof}
\begin{theorem}
	For any function $h$ that satisfies conditions of Theorem \ref{Cauchy_problem_FP}, the function
	\begin{equation*}
		g_{\alpha, \gamma}^{H}(t,y)=\mathbb{E} \left[g_{1, 0}^{H}(Z^{(\alpha, \gamma)}_t,y)\right]
	\end{equation*}
	solves the fractional Cauchy problem \eqref{Cauchy_problem2}, where $g_{\alpha, \gamma}^{H}(t,y)$ is given by \eqref{Cauchy_problem_solution2}.
\end{theorem}
\begin{proof}
The proof follows similar lines as the previous theorem and is therefore omitted.
\end{proof}

\begin{remark}
Results such as last two theorems are typical for Cauchy problems involving Caputo fractional derivative, where inverse of subordinators are the main tool for connecting analytical solutions with their stochastic counterparts. However, given that our study employs a modified form of the Caputo fractional derivative, these results warrant further attention. Moreover, our approach can be seen as a generalization since $\gamma=0$ in $	\operatorname{\mathcal{D}}^{(\alpha,\gamma)}_t$ leads to recovering Caputo fractional derivative. An intriguing open question is whether well-established results related to fractional Cauchy problems—such as the subordination principle (\cite{Bajlekova}) and the time-changed semigroup property (\cite{BaeumerMeerschaert})—remain valid in our modified framework.
\end{remark}

\begin{remark}
To the best of our knowledge, the non-fractional Cauchy problem 
\begin{equation*} 
	A\cdot \frac{d^2}{dt^2} g(t,y)+B \cdot \frac{d}{dt} g(t,y)= \mathcal{G} g(t,y), \quad g(0,y)=h(y)
\end{equation*}
does not admit a direct stochastic representation. This limits the interpretability of the last two theorems in contrast to the case of stretched non-local Pearson diffusions.
\end{remark}
\appendix
\section{Double gamma function} \label{Appen}

\renewcommand{\theequation}{A.\arabic{equation}}
\setcounter{equation}{0}

The double gamma function, denoted by $G(z;\tau)$, can be defined as \cite{Kuznetsov}
\begin{equation}
	\label{ap.B.G.tau.2}
	G(z;\tau) = \frac{1}{\tau\Gamma(z)} \operatorname{e}^{\left[\tilde{a}(\tau)\frac{z}{\tau} + \tilde{b}(\tau) 
		\frac{z^2}{2\tau^2}\right]} \prod_{m=1}^\infty \frac{\Gamma(m\tau)}{\Gamma(z+m\tau)} 
	\operatorname{e}^{\left[z\psi(m\tau)+ \frac{z^2}{2}\psi^\prime(m\tau) \right]} ,
\end{equation}
where $z \in \mathbb{C}$, $\tau \in \mathbb{C}\setminus (-\infty,0]$ and 
\begin{equation*}
	\tilde{a}(\tau) = a(\tau) - \gamma\tau , \qquad \tilde{b}(\tau) = b(\tau) + \frac{\pi^2 \tau^2}{6} ,
\end{equation*}
with $a(\tau)$ and $b(\tau)$ given by 
\begin{equation*}
	\begin{split}
		& a(\tau) = \gamma \tau + \frac{\tau}{2}\log{(2\pi \tau)} + \frac{1}{2}\log\tau - \tau C(\tau) , \\
		& b(\tau) = - \frac{\pi^2 \tau^2}{6} - \tau \log\tau - \tau^2 D(\tau) , 
	\end{split}
\end{equation*}
and 
\begin{equation*}
	\begin{split}
		& C(\tau) = \lim_{m\to \infty} \left[\sum_{k=1}^{m-1}\psi(k\tau) + \frac{1}{2}\psi(m\tau) - \frac{1}{\tau} 
		\log\left(\frac{\Gamma(m\tau)}{\sqrt{2\pi}}\right) \right] , \\
		& D(\tau) = \lim_{m\to \infty} \left[\sum_{k=1}^{m-1}\psi^\prime(k\tau) + \frac{1}{2}\psi^\prime(m\tau) - \frac{1}{\tau}\psi(m\tau) \right] . 
	\end{split}
\end{equation*}
Note that $G(z;\tau)$ is an entire function with zeros located at $-(\mu\tau + \lambda)$ with $\lambda,\mu= 0,1,2,\ldots$.

The double gamma function is such that \cite{Genesis} 
\begin{equation*}
	\label{ap.B.G.tau.(1)}
	G(1;\tau) = 1 . 
\end{equation*}
It satisfies the functional relations 
\begin{equation}
	\label{ap.B.general.G.tau}
	G(z+1;\tau) = \Gamma(z/\tau) G(z;\tau) 
\end{equation} 
and 
\begin{equation*}
	\label{ap.B.general.G.tau.2}
	G(z+\tau;\tau) = (2\pi)^{\frac{\tau-1}{2}} \tau^{-z+\frac{1}{2}} \Gamma(z) G(z;\tau) .
\end{equation*}
A straightforward consequence of these properties is 
\begin{equation*}
	\label{ap.B.G.tau.cons}
	G(1+\tau;\tau) = G(\tau;\tau) = (2\pi)^{(\tau-1)/2} \tau^{-1/2} . 
\end{equation*}
Moreover, using \eqref{ap.B.general.G.tau} recursively, we obtain
\begin{equation}
	\label{ap.B.general.G.tau.n}
	G(z+k;\tau) = G(z,\tau) \prod_{j=0}^{k-1} \Gamma[(z+j)/\tau] . 
\end{equation}

The Stirling formula for $G(z;\tau)$ is (see \cite{Kuznetsov}) 
\begin{equation}
	\label{ap.B.Stirling.G.tau}
	\log{G(z;\tau)} = \left[a_2(\tau) z^2 + a_1(\tau)z + a_0(\tau)\right] \log{z} + 
	b_2(\tau)z^2 + b_1(\tau) z + b_0(\tau) + \mathcal{O}(z^{-1}) , 
\end{equation}
where 
\begin{equation*}
	\begin{split}
		& a_2(\tau) = \frac{1}{2\tau} , \\
		& a_1(\tau) = -\frac{1}{2}\left(1+\frac{1}{\tau}\right) , \\
		& a_0(\tau) = \frac{\tau}{12} + \frac{1}{4} + \frac{1}{12\tau} , 
	\end{split}
\end{equation*}
and 
\begin{equation*}
	\begin{split}
		& b_2(\tau) = -\frac{1}{2\tau}\left(\frac{3}{2} + \log\tau\right) , \\
		& b_1(\tau) = \frac{1}{2}\left(\left(1+\frac{1}{\tau}\right)(1+\log\tau) + \log{2\pi}\right) , \\
		& b_0(\tau) = \frac{1}{3}\big\{\log\left[G^2(1/2;\tau)G(\tau;2\tau)\right] \\
		& \phantom{b_0(\tau) = \frac{1}{3}\big\{} - 
		\frac{1+\tau}{2}\log{2\pi} - a_0(\tau)\log{(\tau^3/2)} - \log{2}\big\} . 
	\end{split}
\end{equation*}

\bigskip \bigskip

\textbf{Acknowledgements} \newline

Luisa Beghin acknowledges financial support under NRRP, Mission 4, Component 2, Investment 1.1, Call for tender No. 104 published on 2.2.2022 by the Italian MUR, funded by the European Union – NextGenerationEU– Project Title “Non–Markovian Dynamics and Non-local Equations” – 202277N5H9 - CUP: D53D23005670006.

Nikolai Leonenko (NL) would like to thank for support and hospitality during
the programme “Fractional Differential Equations” and the programmes “Uncertainly Quantification and Modelling of Materials” and “Stochastic systems for anomalous diffusion” in Isaac Newton Institute for Mathematical Sciences, Cambridge. The last programme was organized with the support of the Clay Mathematics Institute, of EPSRC (via grants EP/W006227/1 and EP/W00657X/1), of UCL (via the MAPS Visiting Fellowship scheme) and of the Heilbronn Institute for Mathematical Research (for the Sci-Art Contest). Also NL was partially supported under the ARC Discovery Grant DP220101680 (Australia), Croatian Scientific Foundation (HRZZ) grant “Scaling in Stochastic Models” (IP-2022-10-8081), grant FAPESP 22/09201-8 (Brazil) and
the Taith Research Mobility grant (Wales, Cardiff University). Also, NL would like to thank University of Rome “La Sapienza” for hospitality as Visiting Professor (June 2024) where the paper was initiated.

Ivan Papić was partially supported by the Croatian Science Foundation (HRZZ) Grant Scaling in Stochastic Models (IP-2022-10-8081).

Jayme Vaz would like to thank the support of FAPESP (project 24/17510-6), the
Taith Research Mobility grant (Wales, Cardiff University), and Cardiff University for the hospitality during the completion of this paper.

\newpage

\bibliographystyle{plain}
\bibliography{Streched_non_local_Pearson_diffusions}

\end{document}